\documentclass{amsart}
\usepackage{eurosym}
\usepackage{amsmath}
\usepackage{amsfonts,mathrsfs}
\usepackage{graphicx}
\usepackage{amsfonts}
\usepackage{amssymb}
\usepackage{xcolor}
\usepackage{exscale}
\usepackage{latexsym}
\usepackage{esint,cleveref}
\newcommand{\bfs}{\boldsymbol}

\newtheorem{theorem}{Theorem}[section]

\newtheorem{corollary}[theorem]{Corollary}

\newtheorem{lemma}[theorem]{Lemma}

\newtheorem{proposition}[theorem]{Proposition}
\newtheorem{remark}[theorem]{Remark}

\numberwithin{equation}{section}

\mathchardef\mhyphen="2D

\usepackage{bm}
\usepackage{bbm}

\begin{document}

\title[$L^{p}$-solvability of nonlocal system]{A note on the $L^{p}$-solvability of a strongly-coupled nonlocal system of equations}

\author[Miriam Abbate]
{Miriam Abbate}
\address{Miriam Abbate, Department of Mathematics, University of Tennessee Knoxville, Knoxville, TN 37996}
\email{mabbate3@vols.utk.edu}
\author[Tadele Mengesha]
{Tadele Mengesha}
\address{Tadele Mengesha, Department of Mathematics,  University of Tennessee Knoxville, Knoxville, TN 37996}
\email{mengesha@utk.edu}

\everymath{\displaystyle}

\maketitle
\begin{abstract}
The goal of this paper is to study the $L^p$-solvability of the strongly-coupled nonlocal system 
\begin{equation*}
     \mathbb{L} \bm{u} (\bm{x})  + \lambda \bm{u}(\bm{x})= \bm{f}(\bm{x}) \quad \text{in $\mathbb{R}^{d}$ }
\end{equation*}
where $\mathbb{L}$ is a linear nonlocal coupled vector-valued operator associated with a kernel $K$ comparable to $|\bm{y}|^{-(d+2s)}$ for $s \in (0,1)$, satisfying certain ellipticity and cancellation conditions. 
For any $\bm{f} \in [L^p(\mathbb{R}^d)]^d$, $1< p < \infty$, the existence of a unique strong solution $\bm{u} \in [H^{2s,p}(\mathbb{R}^d)]^d$ is proved via the method of continuity. To apply this method, we establish the continuity of the operator $\mathbb{L}$ and the necessary \textit{a priori} estimates. 
These are obtained through the study of the corresponding parabolic system. The proof strategy follows and extends recent ideas developed for the scalar setting, combining commutator estimates, Sobolev embeddings, a level set estimates and a bootstrap argument. 
\end{abstract}

\keywordsname: $L^{p}$-estimates, strongly-coupled nonlocal systems,  optimal regularity, Peridynamics 
\section{Introduction and statement of the main result}
\subsection{Motivation and set up of the problem}
In this work, we investigate the $L^{p}$-solvability of the strongly-coupled system of linear equations:
\begin{equation}
\label{system}
     \mathbb{L} \bm{u} (\bm{x})  + \lambda \bm{u}(\bm{x})= \bm{f}(\bm{x}) \quad \text{in $\mathbb{R}^{d}$ }
\end{equation}
where $\mathbb{L}$ is a nonlocal vector-valued operator defined as
\begin{align*}
- \mathbb{L}\bm{u}(\bm{x}) =  \int_{\mathbb{R}^d} \left( \frac{\bm{y} \otimes \bm{y}}{|\bm{y}|^2} \right) \left( \bm{u}(\bm{x} + \bm{y}) - \bm{u}(\bm{x}) - D[\bm{u}](\bm{x}) \bm{y} \chi^{(s)}(\bm{y}) \right) K(\bm{y}) \ d\bm{y}. 
\end{align*}
The notations we used to define the operator will be explained later along with precise conditions on $K$ but for now it suffices to say that $K(\bm{y}) $ is bounded from above and below by a positive constant multiple of the  fractional kernel $|\bm{y}|^{-(d+ 2s)}$ for  $s\in (0, 1)$. Equation \eqref{system} has been studied in \cite{MENGESHA2020123919}.   In particular, corresponding to $K$ in  a broad class of kernels,  existence of a unique solution is shown for the case $p=2$. In addition, in the special case of $K(\bm{y}) = |\bm{y}|^{-(d+ 2s)}$, using a Fourier multiplier theorem,  it is shown that \eqref{system} has a unique solution in the Sobolev space $[H^{2s, p}(\mathbb{R}^{d})]^{d}$ for any $1<p<\infty,$ see below for the definition of the space.   In this paper, we extend this result to general $K$ that is comparable to $|\bm{y}|^{-(d+ 2s)}$. Namely, we demonstrate  that for any $\lambda >0$, $p\in (1, \infty)$, and $ {\bm f}\in [L^{p}(\mathbb{R}^{d})]^{d}$, the system \eqref{system} has  a unique  strong solution $\bm{u}$ that belongs to $[H^{2s, p}(\mathbb{R}^{d})]^{d}$.   

As discussed in \cite{MENGESHA2020123919} one uses the solvability of \eqref{system} together with the theory of semigroups to establish well posedness of the wave-type system of time dependent problem 
\begin{equation}
\label{system2}
\begin{cases}
    \partial_{tt} \bm{u}(t,\bm{x}) + \mathbb{L} \bm{u} (t,\bm{x}) = \bm{f}(t,\bm{x}), &\qquad (t,\bm{x}) \in [0,T] \times \mathbb{R}^d, \\
    \bm{u}(0,\bm{x}) = \bm{u}_0(\bm{x}), &\qquad \bm{x} \in \mathbb{R}^d, \\
    \partial_t\bm{u}(0, \bm{x}) = \bm{v}_0(\bm{x}), &\qquad \bm{x} \in \mathbb{R}^d. 
\end{cases}
\end{equation}
The above wave-type system of evolution equations is motivated by the equation of motion in linearized bond-based peridynamics \cite{silling2010linearized,silling2000reformulation}. In this framework, an elastic material is modeled as a mass-spring system, where interactions occur between all pairs of points $\bm{x}$ and $\bm{y}$ within the body via the bond vector $\bm{x} - \bm{y}$. Under external loading $\bm{f}$, the material undergoes a displacement $\bm{u}$, such that the deformed position is given by $\bm{x} + \bm{u}(t, \bm{x})$. In the uniform small-strain theory, the bond strain and the internal linearized force density function are expressed by the linearized strain 
$
(\bm{u}(t, \bm{x}) - \bm{u}(t, \bm{y})) \cdot \frac{\bm{x} - \bm{y}}{|\bm{x} - \bm{y}|} 
$ and $\int_{\mathbb{R}^d}  \left( (\bm{u}(t,\bm{y})-\bm{u}(t,\bm{x}) ) \cdot \frac{\bm{y} - \bm{x}}{|\bm{y} - \bm{x}|} \right) \frac{\bm{y} - \bm{x}}{|\bm{y} - \bm{x}|} \  \rho(\bm{x}, \bm{y}) \ d\bm{y}$, respectively,  
where the kernel $\rho : \mathbb{R}^d \times \mathbb{R}^d \to \mathbb{R}$ encodes the strength and range of interactions among the material points. 
The peridynamic balance of forces in this setting gives rise to the strongly-coupled system of equations
\begin{equation}
\label{nonstst}
\partial_{tt} \bm{u}(t,\bm{x}) +  \int_{\mathbb{R}^d}  \left( (\bm{u}(t,\bm{x}) - \bm{u}(t, \bm{y})) \cdot \frac{\bm{x} - \bm{y}}{|\bm{x} - \bm{y}|} \right) \frac{\bm{x} - \bm{y}}{|\bm{x} - \bm{y}|} \  \rho(\bm{x}, \bm{y}) \ d\bm{y} = \bm{f}(t,\bm{x}),
\end{equation}
The dependence of $\rho$ on $\bm{x}$ and $\bm{y}$ can vary. For example, in the case of homogeneous isotropic
materials, $\rho$ depends only on the relative distance $|\bm{x} - \bm{y}|$. More general kernels $\rho$ have the potential to model
heterogeneous and anisotropic long-range interactions. We focus on models that use a translation invariant $\rho(\bm{x}, \bm{y}) = K(\bm{x}-\bm{y})$ where $K$ is not necessarily even.  In this case,  the internal linearized force density function can be rewritten as 
\[
\int_{\mathbb{R}^d}  \left( (\bm{u}(t, \bm{y} +\bm{x})-\bm{u}(t,\bm{x}) ) \cdot \frac{\bm{y}}{|\bm{y}|} \right) \frac{\bm{y}}{|\bm{y}|} \  K(\bm{y}) \ d\bm{y} = \int_{\mathbb{R}^d} \left( \frac{\bm{y} \otimes \bm{y}}{|\bm{y}|^2} \right) \left( \bm{u}(t, \bm{x} + \bm{y}) -\bm{u}(t,\bm{x}) \right) K(\bm{y}) \ d\bm{y}. 
\]
In the event that $K\in L^{1}(\mathbb{R}^{d})$, the above integral converges absolutely for compactly supported smooth functions.  If $K$ is not integrable, we must modify the integrand to get a well defined quantity. To that end, we work with a class of kernels that are fractional-type. That is, we assume that there are constants $\alpha_1, \alpha_2 \in (0, \infty)$ such that for any $s\in (0, 1)$
\begin{equation}\label{Ellipticity}
(1-s)\alpha_1 \leq K(\bm{y})|\bm{y}|^{-(d+2s)} \leq (1-s)\alpha_2 \quad \forall \ \bm{y}\neq {\bfs 0}. 
\end{equation}
Taking advantage of the presence of the even rank-one matrix $ \frac{\bm{y} \otimes \bm{y}}{|\bm{y}|^2} $, we modify the above integral  to obtain the operator 
\begin{align}\label{leading-operator}
- \mathbb{L}\bm{u}(\bm{x}) =  \int_{\mathbb{R}^d} \left( \frac{\bm{y} \otimes \bm{y}}{|\bm{y}|^2} \right) \left( \bm{u}(\bm{x} + \bm{y}) - \bm{u}(\bm{x}) - D[\bm{u}](\bm{x}) \bm{y} \chi^{(s)}(\bm{y}) \right) K(\bm{y}) \ d\bm{y}. 
\end{align}
with $D[\bm{u}](\bm{x})$ denoting the symmetric part of the gradient:
$
D[\bm{u}](\bm{x}) = \frac{1}{2} \left(\nabla \bm{u}(\bm{x}) + \nabla \bm{u}(\bm{x})^T\right), 
$ 
and 
\[
\chi^{(s)}(\bm{y}) = 
\begin{cases}
0 & \text{if } s \in (0, 1/2), \\
\mathbbm{1}_{B_1}(\bm{y}) & \text{if } s = 1/2, \\
1 & \text{if } s \in (1/2, 1), 
\end{cases}
\]
where $B_{r}$ denotes a ball of radius $r$ centered at the origin.   
For $s=\frac{1}{2}$, we also assume that 
\begin{equation}\label{cancellation}
\int_{\partial B_r} \bm{y} K(\bm{x}, \bm{y}) \ dS_r(\bm{y}) = {\bfs 0},
\end{equation}
which ensures that, in the definition of $\chi^{(s)} (\bm{y}) $, $\mathbbm{1}_{B_1}(\bm{y})$  can be replaced by $\mathbbm{1}_{B_r}(\bm{y})$ for any $r>0$. 
With this new modification, it is clear that if $K$ satisfies \eqref{Ellipticity} and \eqref{cancellation}, the integral converges absolutely for smooth $\bm{u}$. In the event $K$ is even, \eqref{cancellation} is automatically satisfied. For $K(\bm{y}) = c(d,s)|\bm{y}|^{-(d+2s)},$ the operator in \eqref{leading-operator} becomes 
\begin{equation}\label{fractional-Lame}
- (-\mathring{\Delta}^{s})\bm{u}(\bm{x}) = c(d,s)\int_{\mathbb{R}^d} {1 \over |\bm{y}|^{d+2s}}\left( \frac{\bm{y} \otimes \bm{y}}{|\bm{y}|^2} \right) \left( \bm{u}(\bm{x} + \bm{y}) - \bm{u}(\bm{x}) - D[\bm{u}](\bm{x}) \bm{y} \chi^{(s)}(\bm{y}) \right)  \ d\bm{y}
\end{equation}
which is considered the fractional analog of the operator Lam\'e in linearized elasticity, \cite{MD-2014}.

\subsection{Statement of the main results} 
The main objective of this work is to study the $L^{p}$-solvability of \eqref{system} with the leading operator defined in \eqref{leading-operator} under assumptions \eqref{Ellipticity} and \eqref{cancellation}.  
Before stating the main results, we introduce the necessary notations to define the relevant function spaces. The Fourier transform $\mathcal{F}$ of a vector field $\bm{f}$ is defined by
\[
\mathcal{F}\bm{f}(\bm{\xi}) = \widehat{\bm{f}}(\bm{\xi}) := \int_{\mathbb{R}^d} \bm{f}(\bm{x}) \ e^{-2\pi i \bm{x} \cdot \bm{\xi}} \ d\bm{x}.
\]
Let $[\mathcal{S}]^d$ denote the space of $\mathbb{R}^d$-valued Schwartz vector fields, and let $[\mathcal{S}']^d$ denote its dual. For $p \in (1, \infty)$ and $s \in (0,1)$, we define the Bessel potential space
\[
[H^{2s,p}(\mathbb{R}^d)]^d := \left\{ \bm{u} \in [\mathcal{S}']^d \Bigg|\ \left((1 + 4\pi^2 |\bm{\xi}|^2)^s \widehat{\bm{u}}\right)^\vee \in [L^p(\mathbb{R}^d)]^d \right\},
\]
with norm
\[
\|\bm{u}\|_{[H^{2s,p}(\mathbb{R}^d)]^d} := \left\| \left((1 + 4\pi^2 |\bm{\xi}|^2)^s \widehat{\bm{u}}\right)^\vee \right\|_{[L^p(\mathbb{R}^d)]^d}.
\]
The corresponding homogeneous space is given by
\[
[\dot{H}^{2s,p}(\mathbb{R}^d)]^d := \left\{ \bm{u} \in [\mathcal{S}']^d \Bigg|\ \left((4\pi^2 |\bm{\xi}|^2)^s \widehat{\bm{u}}\right)^\vee =: (-\Delta)^s \bm{u} \in [L^p(\mathbb{R}^d)]^d \right\},
\]
with associated seminorm
\[
[\bm{u}]_{[H^{2s,p}(\mathbb{R}^d)]^d} := \|(-\Delta)^s \bm{u}\|_{[L^p(\mathbb{R}^d)]^d}.
\]
We are now ready to state our main result.

\begin{theorem}
\label{mainthr}
Let $s \in (0,1)$, $\lambda > 0$, and $p \in (1,\infty)$. Suppose that the kernel $K$ satisfies \eqref{Ellipticity} and \eqref{cancellation}. Then the operator $\mathbb{L}$, defined in \eqref{leading-operator}, has a unique continuous extension from $[H^{2s,p}(\mathbb{R}^d)]^d$ to $[L^p(\mathbb{R}^d)]^d$. Moreover, for every $\bm{f} \in [L^p(\mathbb{R}^d)]^d$, there exists a unique solution $\bm{u} \in [H^{2s,p}(\mathbb{R}^d)]^d$ to the equation
\[
\mathbb{L}\bm{u} + \lambda \bm{u} = \bm{f} \quad \text{in } \mathbb{R}^d,
\]
which satisfies the estimates
\begin{equation}
\label{est1}
\|(-\Delta)^s \bm{u}\|_{[L^p(\mathbb{R}^d)]^d} + \lambda \|\bm{u}\|_{[L^p(\mathbb{R}^d)]^d} \leq N \|\bm{f}\|_{[L^p(\mathbb{R}^d)]^d},
\end{equation}
for a constant $N> 0$  that depends only on $d$, $s$, $p$, $\alpha_1$, $\alpha_2$. As a consequence, 
\begin{equation}
\label{est2}
\|\bm{u}\|_{[H^{2s,p}(\mathbb{R}^d)]^d} \leq C \|\bm{f}\|_{[L^p(\mathbb{R}^d)]^d},
\end{equation}
for $C > 0$  depending only on $d$, $s$, $p$, $\alpha_1$, $\alpha_2$ and $\lambda$.
\end{theorem}
\begin{remark}
Estimate \eqref{est2} follows directly from \eqref{est1}, since
\[
\|\bm{u}\|_{[H^{2s,p}(\mathbb{R}^d)]^d} \approx \|\bm{u}\|_{[L^p(\mathbb{R}^d)]^d} + \|(-\Delta)^s \bm{u}\|_{[L^p(\mathbb{R}^d)]^d}.
\]
\end{remark}
For $p=2$, solvability of \eqref{system} is established in \cite{Du_Zhou_2011}.  In that work, the $H^{2s,2}$ regularity of appropriately-defined weak solutions of (1.7) with $\lambda = 1$ and $K$  positive and radially symmetric was obtained via the Fourier transform. See also \cite{Dang2025Regularity} for a similar regularity result for problems in the periodic setting.  Using the method of continuity, the result was extended in \cite{MENGESHA2020123919} to  hold for operators with kernel $K$ that are only translation invariant and can possibly vanish on a substantial set, specifically outside of a double cone whose apex is at the origin.  

For $p\neq 2$, the $L^{p}$-solvability of the nonlocal system \eqref{system} is relatively unstudied, and to the best of our knowledge, the $L^{p}$-solvability in its full generality is not available in the literature.  The only result we are aware of is in \cite{MENGESHA2020123919} where $L^{p}$-solvability is established for the special choice $K(\bm{y}) = c(d, s)|\bm{y}|^{-(d + 2s)}$, with the corresponding operator being $-(-\mathring{\Delta})^{s}$, defined in \eqref{fractional-Lame}. The Fourier matrix symbol of $-(-\mathring{\Delta})^{s}$  can be computed explicitly as \begin{equation}\label{matrix-symbol-forfrac}\mathbb{M}^{\Delta}({\bfs \xi}) = (2\pi |\bm{\xi}|)^{2s} \left(\ell_1 \mathbb{I} + \ell_{2} {{\bfs \xi} \otimes {\bfs \xi}\over | \bm {\xi}|^{2}},  \right)\end{equation} for some positive constants  $\ell_1, \ell_2$ and $\ell_1 \neq \ell_{2}$,  \cite{scott2020thesis,MENGESHA2020123919}.   This particular form of the Fourier symbol allows for the use of Fourier multiplier theorems to demonstrate that for any ${\bm f}\in [L^{p}(\mathbb{R}^{d})]^{d}$ and $\lambda>0$, the solution $\bm{u} = \mathcal{F}^{-1}[(\mathbb{M}^{\Delta}({\bfs \xi}) + \lambda \mathbb{I} )^{-1} \hat{{\bm f}}] \in [H^{2s, p}(\mathbb{R}^{d})]^{d}  $.   For general $K$, the Fourier matrix symbol associated with the operator $\mathbb{L}$ is:
\begin{equation}\label{FMofL}
\mathbb{M}({\bfs \xi}) = \int_{\mathbb{R}^d} 
\frac{\bm{y}\otimes \bm{y}}{|\bm{y}|^2}\left(e^{\imath 2\pi{\bfs \xi}\cdot \bm{y}} - 1 -2\pi \imath{\bfs \xi}\cdot \bm{y} \chi^{(s)}(\bm{y}) \right)K(\bm{y}) d\bm{y},
\end{equation}
which in general lacks differentiability necessary to apply classical multiplier theorems. 
 We should mention that interior $H^{2s, p}$-regularity of weak solutions  to 
the Dirichlet problem for equations resembling \eqref{system}  subject to a complementary condition (also known as “volume constraint problem” in peridynamics) is established in \cite{KassmannMengeshaScott2019}. See also a recent work on the Calderon-Zygmund theory for a strongly-coupled systems of nonlocal equations with kernel that is translation variant and H\"{o}lder continuous \cite{MengeshaSchikorraSeesaneaYeepo}. See also the recent work \cite{defilippis2025partialregularitynonlocalsystems} for fine properties of weak solution of elliptic nonlocal systems. 
In the current work,
in place of weak solutions, we consider strong solutions solving an equation almost everywhere. 

Even for the case of scalar elliptic and parabolic equations corresponding to the nonlocal operator related to $\mathbb{L}$,  namely, 
\begin{equation}\label{scaler-main-op}
\mathcal{L} u(\bm{x}) = -\int_{\mathbb{R}^d} \left( u(\bm{x} + \bm{y}) - u(\bm{x}) - \nabla u(\bm{x})\cdot \bm{y} \chi^{(s)}(\bm{y}) \right) K(\bm{y}) \ d\bm{y}, 
\end{equation}
the result is highly nontrivial and has been the focus of  significant recent research efforts. These problems have been extensively studied in \cite{Mikulevicius1992,Mikulevicius2014,Mikulevicius2019, BanuelosBogdan2007,DongKim2012,Dong2023}, among others. 
The operator $\mathcal{L}$ is a nonlocal elliptic operator associated with a L\'{e}vy process. The solvability of parabolic equations related to the elliptic problem $\mathcal{L} u+\lambda u$ in Sobolev and H\"{o}lder spaces has been investigated in \cite{Mikulevicius1992} using a probabilistic approach. The boundedness of $\mathcal{L}: H^{2s,p} \to L^{p}$ for large $p$  in \cite{Mikulevicius2014} is established with an argument based on the classical theory of singular integrals. Following a similar technique, this is extended in \cite{Dong2023} to hold for arbitrary $p>1$ in the weighted space.     In \cite{BanuelosBogdan2007}, L\'{e}vy processes were analyzed using their Fourier symbols, which are expressed as quotients of symbols that involve integrals.  These representations facilitate the application of Fourier multiplier theorems to derive {\em a priori} estimates for solutions. However, given that the vector-valued operator $\mathbb{L}$ is strongly-coupled with a matrix-valued symbol, the results and techniques of \cite{BanuelosBogdan2007} cannot be directly applied to obtain a priori estimates for solutions to the system of equations \eqref{system}. In contrast,  real analytic techniques were employed in \cite{DongKim2012}, to establish $L^{p}$-solvability of the scalar elliptic equation $ \mathcal{L} u+\lambda u = f$.  This approach relies on maximum principle techniques to demonstrate the existence of solutions as well as  other approaches to show the boundedness of the operator $\mathcal{L}: H^{2s, p}(\mathbb{R}^{d}) \to L^{p}(\mathbb{R}^{d})$. Unfortunately, these techniques are not applicable to the strongly-coupled nonlocal system \eqref{system}, leaving the question of obtaining an analogous $L^{p}$-solvability for this system unresolved.

Our current work is based on and inspired by a recent study \cite{Dong:2023aa} that successfully established {\em a priori} estimates and unique $L^{p}$-solvability for (fractional) parabolic nonlocal equations without relying on maximum principle techniques. The solvability of these nonlocal equations is achieved through the method of continuity, which fundamentally depends on the derivation of {\em a priori} estimates and the continuity of $\mathcal{L}: H^{2s, p}(\mathbb{R}^{d}) \to L^{p}(\mathbb{R}^{d})$. As detailed in \cite{Dong:2023aa}, the proof of a priori estimates employs a combination of advanced techniques, including a level set argument, the “crawling of ink spots lemma” and a bootstrap argument. Specifically, the estimate is first established for $p=2$ using the Fourier transform and then an iterative procedure is applied to extend these results from $p=2$ to $p=\infty$. The continuity of the operator $\mathcal{L}: H^{2s, p}(\mathbb{R}^{d}) \to L^{p}(\mathbb{R}^{d})$ is established through an auxiliary estimate for solutions of parabolic equations. Notably, this approach does not rely on the analysis of Fourier symbols or the use of maximum principle techniques and can be readily adapted to analyze the strongly-coupled system of nonlocal equations \eqref{system}, and the primary objective of this work is to achieve that adaptation. 

This paper is organized as follows. Section \ref{prop-operator} discusses elementary properties of the operator $\mathbb{L}$. We then outline the application of the method of continuity for establishing existence of solutions in Section \ref{M-o-C}, while Section \ref{Cont-and-aPrior} presents the necessary ingredients, including continuity and a priori estimates, for its application. We conclude the paper in Section \ref{System-Parabolic} by proving the existence and stability estimates of solutions for the fractional parabolic system of nonlocal equations, adapting the approach used in \cite{Dong:2023aa}.

\section{Properties of the operator $\mathbb{L}$}\label{prop-operator}
In this section we present some elementary properties of the operator ${\mathbb{L}}.$  We assume that the kernel $K$ satisfies conditions \eqref{Ellipticity} and \eqref{cancellation}. 
We recall that for $\bm{u} \in [\mathcal{S}]^{d}$,  $\mathbb{L}\bm{u}$ is well defined at all $\bm{x} \in \mathbb{R}^{d}$ as an absolutely convergent integral and that 
$
\widehat{\mathbb{L}\bm{u}}({\bfs \xi}) = \mathbb{M}({\bfs \xi}) \widehat{\bm{u}}({\bfs \xi}),  
$ 
where $\mathbb{M}({\bfs \xi})$ is as defined in \eqref{FMofL}. We show first that $\mathbb{L}$ commutes with derivatives and as a consequence for any $\bm{u} \in  [\mathcal{S}]^{d}$, $\mathbb{L}\bm{u}\in [C^{\infty}(\mathbb{R}^{d})]^{d}$. We verify this by showing without loss of generality that 
\[
\mathbb{L}(\partial_{x_1}\bm{u})(\bm{x}) = \lim_{h\to 0} \frac{\mathbb{L}\bm{u}(\bm{x}+ h \bm{e}_1) - \mathbb{L} \bm{u}(\bm{x})}{h}
\]
which says $ \mathbb{L}(\partial_{x_1}\bm{u})(\bm{x}) = \partial_{x_1} (\mathbb{L} \bm{u}(\bm{x}))$. We denote 
\[
\delta_{s}[\bm{u}](\bm{x}, \bm{y}) :=\bm{u}(\bm{x} + \bm{y}) - \bm{u}(\bm{x}) - D[\bm{u}](\bm{x}) \bm{y} \chi^{(s)}(\bm{y}) 
\]
Then 
\begin{align*}
\lim_{h\to 0} \frac{\mathbb{L}\bm{u}(\bm{x}+ h \bm{e}_1) - \mathbb{L} \bm{u}(\bm{x})}{h}= - \lim_{h\to 0}  \int_{\mathbb{R}^d}   \left( \frac{\bm{y} \otimes \bm{y}}{|\bm{y}|^2} \right) {\delta_{s}[\bm{u}](\bm{x} + h{\bm e}_1, \bm{y})-\delta_{s}[\bm{u}](\bm{x}, \bm{y})\over h} K(\bm{y}) \ d\bm{y}. 
\end{align*}
If we can interchange the limit and the integration then, since $\bm{u}$ is smooth, 
\[
\lim_{h\to 0} \frac{\mathbb{L}\bm{u}(\bm{x}+ h \bm{e}_1) - \mathbb{L} \bm{u}(\bm{x})}{h} = -  \int_{\mathbb{R}^d} \left( \frac{\bm{y} \otimes \bm{y}}{|\bm{y}|^2} \right) \delta_{s}[\partial_{x_1}\bm{u}](\bm{x}, \bm{y}) K(\bm{y}) \ d\bm{y} = \mathbb{L}(\partial_{x_1}\bm{u})(\bm{x}). 
\]
Interchanging of limit and integration is justified by applying dominated convergence theorem after noting that  
\[
\left| {\delta_{s}[\bm{u}](\bm{x} + h{\bm e}_1, \bm{y})-\delta_{s}[\bm{u}](\bm{x}, \bm{y})\over h} K(\bm{y})\right| \leq 
\left\{ \begin{split}
&C_1(\bm{u}) |\bm{y}|^{1-d}\quad \text{for $|\bm{y}| \leq 1$}\\
&C_{2} (\bm{x}, \bm{u}) \left({1\over |\bm{y}|^{d + 2s}} + \frac{1}{|\bm{y}|^{d+2s-1}} \chi^{(s)}(\bm{y})\right)\quad \text{for $|\bm{y}| \geq 1$}. 
\end{split}
\right.
\]
As in the fractional Laplacian case, see \cite{Silvestre2007Regularity}, it can be shown that for any $k$ nonnegative integer, $(1 + |\bm{x}|^{d +2s})\partial^{k}\mathbb{L}(\bm u)$ is bounded in $\mathbb{R}^{d}$, and  $\mathbb{L}\bm{u} \in [\bar{\mathcal{S}_{s}}]^{d}$ whenever $\bm{u} \in [\mathcal{S}]^d$. Here 
\[
\bar{\mathcal{S}_{s}} =\left\{ v\in C^{\infty}(\mathbb{R}^{d}): \forall k\geq 0, k \ \text{integer}, \sup_{\bm{x}\in \mathbb{R}^{d}}((1 + |\bm{x}|^{d +2s}) \partial^{k} v(\bm{x})) < \infty \right\}, 
\]
equipped with the semi norms $[ v ]_k = \sup(1 + |\bm{x}|^{d+2s}) \partial^{k} v(\bm{x})$.  We denote the dual of $\bar{\mathcal{S}_{s}}$ by $\bar{\mathcal{S}'_{s}}$. 
We extend the definition of $\mathbb{L}$ to the space $[\bar{\mathcal{S}'_{s}}]^{d}$ by duality as 
\[
\langle \mathbb{L}\bm{u}, \bm{v}\rangle = \langle \bm{u}, \mathbb{L}^{\ast}\bm{v} \rangle
\]
where $\mathbb{L}^{\ast}$ is an operator of the same form as $\mathbb{L}$ associated with kernel $K^* (\bm{y}) = K(-\bm{y})$. In the event $K$ is radial, then two operators, $\mathbb{L}^{\ast}$ and $\mathbb{L}$,  are the same.  Moreover, $K^* $ will satisfy the same ellipticity and cancellation conditions, \eqref{Ellipticity} and \eqref{cancellation}.  It follows that
$\mathbb{L}$ is a continuous operator from $[\bar{\mathcal{S}'_{s}}]^{d}$ to $[\mathcal{S'}]^{d}$. 

Let $\psi(\bm{x}) = {1\over 1+ |\bm{x}|^{d+2s}}$. 
It is proven in \cite[Propositon 2.4]{Silvestre2007Regularity} that if $\bm{u}\in [L^{1}(\mathbb{R}^{d}, \psi)]^d$ and for some $\epsilon >0,$  $\bm{u} \in [C^{0, 2s + \epsilon}(\mathbb{R}^{d})]^d$ if $s\in (0, 1/2]$ and in   $[C^{1, 2s + \epsilon-1}(\mathbb{R}^{d})]^d$ if $s\in (1/2, 1)$, then $(-\Delta)^{s} \bm{u}$ is a continuous function and its values are given by the integral representation.   The same result can be established for $\mathbb{L}\bm{u}$. See similar result in \cite{scott2022paper}. 
\begin{proposition} 
\label{property1}
Let $\bm{u}$ belong to $[L^1(\mathbb{R}^d,\psi)]^d$ and additionally, for some $\epsilon>0$, $\bm{u} \in [C^{0, 2s + \epsilon}(\mathbb{R}^{d})]^d$ if $s\in (0, 1/2]$ and in   $[C^{1, 2s + \epsilon-1}(\mathbb{R}^{d})]^d$ if $s\in (1/2, 1)$.
Then $\mathbb{L}\bm{u}$ is continuous in $\mathbb{R}^d$ and for every $\bm{x}\in\mathbb{R}^d$ its value is given by \eqref{leading-operator}. Moreover, if $\bm{u} \in [C_b^{\infty}(\mathbb{R}^{d})]^{d}, $ the integral  \eqref{leading-operator} absolutely converges and $\mathbb{L}\bm{u}\in [C_b^{\infty}(\mathbb{R}^{d})]^{d}$.
\end{proposition}
The proof mimics that given in \cite{Silvestre2007Regularity}; but for completeness we present it in the appendix. Finally, we have the useful result that $\mathbb{L}$ commutes with $(-\Delta)^{s}$. 
\begin{proposition}\label{commute}
Suppose that $\bm{u} \in [C_b^{\infty}(\mathbb{R}^{d})]^{d}$. Then for any $\bm{x}\in \mathbb{R}^{d}$ we have   \[\mathbb{L}((-\Delta)^{s}\bm{u} )(\bm{x}) = (-\Delta)^{s}(\mathbb{L} \bm{u}) (\bm{x}).\]
\end{proposition}
\begin{proof}
The proof is a tedious interchanging of integrals. The main point is by Proposition \ref{property1}, since $\bm{u} \in [C_b^{\infty}(\mathbb{R}^{d})]^{d}$, the integral representations of both $\mathbb{L}\bm{u}(\bm{x})$ and $(-\Delta)^{s}\bm{u}(\bm{x})$ absolutely converge and are in $[C_b^{\infty}(\mathbb{R}^{d})]^{d}$. Thus both compositions $\mathbb{L}((-\Delta)^{s}\bm{u} )(\bm{x}) $ and  $(-\Delta)^{s}(\mathbb{L} \bm{u}) (\bm{x})$ are represented by absolutely convergent double integrals. Interchanging  of integrals is justified by application of Fubini's theorem since the resulting integrals can bounded by 
a sum of functions in $L^{1}(\mathbb{R}^{d}\times \mathbb{R}^{d}, d{\bm z} \, d\bm{y})$.
\end{proof}
In the above and in what follows, $ [C_b^{\infty}(\mathbb{R}^{d})]^{d}$  (resp. $[C_0^{\infty}(\mathbb{R}^{d})]^{d}$) denotes the space of $\mathbb{R}^d$- valued $C^{\infty}$-vector fields whose components and all their partial derivatives are uniformly bounded (res. vanish at $\infty$) on $\mathbb{R}^d.$  

\begin{proposition}
\label{commutestar}
Let $s \in (0, 1)$, $p \in (1,\infty)$. Then $\mathbb{L}$ commutes with the convolution operator in $[\mathcal{S}]^d$. So does its unique continuous extension $\mathbb{L}: [H^{2s,p}(\mathbb{R}^{d})]^{d} \to [L^{p}(\mathbb{R}^{d})]^d$, if it exists.
\end{proposition}
\begin{proof}
Let $\bm{u}\in [\mathcal{S}]^d$ and $v \in L^{1}(\mathbb{R}^{d})$. Then $\bm{u} \ast v \in [\mathcal{S}]^d$ and 
\begin{align*}
\mathcal{F}[\mathbb{L} (\bm{u} \ast v)] (\bm{\xi}) =  \mathbbm{M}(\bm{\xi}) \mathcal{F}[\bm{u} \ast v](\bm{\xi})
=  \mathbbm{M}(\bm{\xi}) \Hat{\bm{u}}(\bm{\xi})  \Hat{v}(\bm{\xi})
=  \mathcal{F}[\mathbb{L} (\bm{u}) ] (\bm{\xi}) \Hat{v}(\bm{\xi}) 
= \mathcal{F}[(\mathbb{L} (\bm{u})) \ast  v ] (\bm{\xi}).
\end{align*}
Given $\bm{u} \in [H^{2s,p}(\mathbb{R}^d)]^d$, there exists a sequence $\{\bm{u}_n\}_n \subset  [\mathcal{S}]^d$ such that $\bm{u}_n \to \bm{u}$ in $[H^{2s,p}(\mathbb{R}^d)]^d$. For each $n$, by the previous computation, $\mathbb{L}( \bm{u}_n \ast v) = \mathbb{L}( \bm{u}_n) \ast v$ for all $v \in \mathcal{S}$. 
Note that 
\begin{align*}
\|\bm{u} \ast v\|_{[H^{2s,p}(\mathbb{R}^d)]^d} = \| ((1+4\pi^2 |\bm{\xi}|^2 )^s \widehat{(\bm{u} \ast v)})^\vee\|_{[L^p(\mathbb{R}^d)]^d} 
&= \| ((1+4\pi^2 |\bm{\xi}|^2 )^s \widehat{\bm{u}} \  \widehat{v})^\vee\|_{[L^p(\mathbb{R}^d)]^d}\\
&\leq \|v\|_{L^1(\mathbb{R}^d)} \|\bm{u} \|_{[H^{2s,p}(\mathbb{R}^d)]^d}
\end{align*}
and, since $\bm{u}_n \to \bm{u}$ in $[H^{2s,p}(\mathbb{R}^d)]^d$, it follows that $\bm{u}_n \ast v \to \bm{u} \ast v$ in $[H^{2s,p}(\mathbb{R}^d)]^d$ for all $v \in \mathcal{S}$. Therefore, on the one hand, 
\begin{align*}
\|\mathbb{L}( \bm{u}_n \ast v)   - \mathbb{L}( \bm{u} \ast v)  \|_{[L^p(\mathbb{R}^d)]^d} &\leq  C \| [(\bm{u}_n -\bm{u} ) \ast v]\|_{[L^p(\mathbb{R}^d)]^d}\leq C \| (\bm{u}_n -\bm{u} ) \ast v\|_{[H^{2s,p}(\mathbb{R}^d)]^d} \to 0
\end{align*}
as $n\to \infty$.  
On the other hand, since $\bm{u}_n \to \bm{u}$ in $[H^{2s,p}(\mathbb{R}^d)]^d$, $\mathbb{L}$ is assumed to have a continuous extension, we have  $\mathbb{L}( \bm{u}_n) \to \mathbb{L}( \bm{u})$ in $[L^p(\mathbb{R}^d)]^d$. It follows that $\mathbb{L}( \bm{u}_n) \ast v  \to \mathbb{L}( \bm{u}) \ast v $ in $[L^p(\mathbb{R}^d)]^d$, as $n\to \infty$. By uniqueness of limit, $\mathbb{L}( \bm{u} \ast v)  = \mathbb{L}( \bm{u}) \ast v$, completing the proof. 
\end{proof}
 
\section{Existence of a solution via method of continuity}\label{M-o-C}
Following \cite{Dong:2023aa} , the main tool we use to prove $L^{p}$-solvability of the nonlocal system \eqref{system} is the method of continuity whose proof can be found in \cite{Jost2007}.

\begin{proposition}
\label{moc}
Let $X$ be a Banach space and $V$ a normed vector space. Suppose that 
\begin{itemize}
\item[i)]
 $\Upsilon_0, \Upsilon_1 : X \to V$ are bounded linear operators
\item [ii)]there exists a constant $C > 0$ such that if $\Upsilon_\tau := (1 - \tau) \Upsilon_0 + \tau \Upsilon_1$ for $\tau \in [0,1]$, then 
\[
\|x\|_X \leq C \|\Upsilon_\tau x\|_V \quad \text{for all } x \in X \text{ and } \tau \in [0,1].
\]
\end{itemize}
Then $\Upsilon_0$ is surjective if and only if $\Upsilon_1$ is surjective.
\end{proposition}
To apply the method of continuity to our case, we take the function spaces  to be $X = [H^{2s, p}(\mathbb{R}^{d})]^{d} $ and $V = [L^{p}(\mathbb{R}^{d})]^{d}$,  and for $\lambda > 0$,  the operators we consider are  $\Upsilon_0 = \alpha_1(-\mathring{\Delta})^{s} + \lambda \mathbb{I}$, where the operator $(-\mathring{\Delta})^{s} $ is defined as in  \eqref{fractional-Lame} and $\alpha_1$ is the ellipticity constant in \eqref{Ellipticity}, and $\Upsilon_1 = \mathbb{L}  + \lambda \mathbb{I}$.  It follows that the convex combination of the operators is 
\[
\Upsilon_\tau := (1 - \tau)\alpha_1(- \mathring{\Delta})^s + \tau\mathbb{L} + \lambda \mathbb{I}, \quad \tau\in [0,1]. 
\]
Let us introduce the  kernel
\[
K_\tau(\bm{y}) := \frac{\alpha_1(1 - \tau)}{|\bm{y}|^{d+2s}} + \tau K(\bm{y}), \quad \text{for $\bm{y} \neq {\bfs 0}$ }. 
\]
The kernel satisfies  the same ellipticity bounds as $K$. In particular, since  $(1-s)\alpha_1 \leq K(\bm{y}) |\bm{y}|^{d+2s} \leq (1-s)\alpha_2$, then
\[
(1-s)\alpha_1 \leq K_\tau(\bm{y}) |\bm{y}|^{d+2s} \leq (1-s)\alpha_2, \quad \text{uniformly in } \tau \in [0,1].
\]
Let $\mathbb{L}_{\tau}$ be the nonlocal operator of the form \eqref{leading-operator}  associated with the kernel $K_{\tau}$. Then $\mathbb{L}_{\tau} = (1 - t)\alpha_1(- \mathring{\Delta})^s + \tau\mathbb{L}$, and $\Upsilon_\tau = \mathbb{L}_{\tau}  + \lambda \mathbb{I}.$

For $1<p<\infty$,  the boundedness of $(-\mathring{\Delta})^{s}:[H^{2s, p}(\mathbb{R}^{d})]^{d}  \to [L^{p}(\mathbb{R}^{d})]^{d}$  is established in \cite[Theorem 6.3]{scott2022paper} using using tools from harmonic analysis. Alternatively, we can show that directly after recalling that for vector fields $\bm{u}$ in the Schwartz space $[\mathcal{S}]^d$, we have 
\begin{equation}\label{Fourier-vector-Lap}
\mathcal{F}((-\mathring{\Delta})^{s}\bm{u}) = (2\pi |{\bfs \xi}|)^{2s}\left(\ell_1\mathbb{I} + \ell_{2} \frac{{\bfs \xi}\otimes {\bfs\xi}}{ |{\bfs \xi}|^2}\right) \mathcal{F}({\bfs u})
\end{equation}
for some positive constants $\ell_1$ and $\ell_2$, $\ell_1 \neq \ell_2,$  depending only on $d$ and $s,$ \cite{scott2020thesis,MENGESHA2020123919}. Then we may write 
\[
\begin{split}
(-\mathring{\Delta})^{s}\bm{u} &= \mathcal{F}^{-1}\left((2\pi |{\bfs \xi}|)^{2s}\left(\ell_1\mathbb{I} + \ell_{2} \frac{{\bfs \xi}\otimes {\bfs\xi}}{ |{\bfs \xi}|^2}\right) \mathcal{F}({\bfs u})\right)\\
&=\ell_1(-\Delta)^{ s} \bm{u} + \ell_2 (\mathcal{R}\otimes \mathcal{R})(-\Delta)^{s}\bm{u}
\end{split}
\] 
where  $\mathcal{R}$ is the Riesz transform with the Fourier symbol $c \imath \frac{\bm \xi}{| \bm\xi|}$. The $L^{p}$-boundedness of $(-\mathring{\Delta})^{s}\bm{u} $ now follows since $\mathcal{R}$  is an $L^{p}$-bounded operator. 
The surjectivity of $\Upsilon_0$ is also proved in \cite{MENGESHA2020123919} (see also \cite[Theorem 6.5]{scott2022paper}) using Fourier multiplier theorems. 

Thus, to apply \Cref{moc} to prove the unique $L^{p}$-solvability of \eqref{system}, it remains to show that $\mathbb{L} :[H^{2s, p}(\mathbb{R}^{d})]^{d}  \to [L^{p}(\mathbb{R}^{d})]^{d}$ is continuous and to establish the {\em a priori} estimates that there exists $C>0$ such that 
\[
\|\bm{u}\|_{[H^{2s, p}(\mathbb{R}^{d})]^d }\leq C \|\Upsilon_{\tau} {\bm u }\|_{[L^{p}(\mathbb{R}^{d})]^d}, \,\text{for all $\bm{u}\in [H^{2s, p}(\mathbb{R}^{d})]^d$ and $\tau\in [0, 1]$.}
\]
These results will be established in the next sections following the real-analytic approach used in \cite{Dong:2023aa}. 
\section{The continuity of $\mathbb{L}$ and {\em a priori} estimates}\label{Cont-and-aPrior}
As previously shown, the continuity of the operator $\mathbb{L}$ is straightforward for case $p=2,$ by using the Fourier transform. Specifically, the continuity follows easily from the estimate of its matrix Fourier symbol,  $|\mathbb{M}({\bfs \xi})| \leq C|{\bfs \xi}|^{2s}$, as detailed in \cite[Theorem 3.2]{MENGESHA2020123919}.  In contrast, establishing the  continuity of $\mathbb{L}$ is nontrivial for $p\neq 2$, a difficulty that persists even for its scalar analog,  $\mathcal{L}$ defined in \eqref{scaler-main-op}. In recent work \cite{Dong:2023aa}, the authors give an indirect proof for the continuity of $\mathcal{L}$ based on an  {\em a priori} estimate for the solutions of fractional parabolic equations \cite[Remark 2.7]{Dong:2023aa}.  We will demonstrate how this indirect approach can be adapted to establish the continuity of $\mathbb{L}$.   To that end, we first state   an  {\em a priori} estimate for solutions of fractional parabolic systems of linear equations.  
Before we state this auxiliary result, let us introduce some notation following \cite{Dong:2023aa}. For $-\infty<T_1<T_2<\infty$ and a given separable Banach space $Y$, the function space
$L^{p}((T_1 ,T_2); Y)$ is  the set of $Y$-valued measurable functions $h$ such that 
\[
\|h\|_{L^{p}((T_1 ,T_2); Y)} :=\left(\int_{T_1} ^{T_2} \|h(t)\|^{p}_{Y} \ dt \right)^{1/p} <\infty. 
\]
If $Y=L^{p}(\mathbb{R}^{d})$, we simply write $L^{p}((T_1 ,T_2)\times \mathbb{R}^{d})$.   For $s\in(0,1)$, we denote $\mathbb{H}^{(2s, p); 1}(T_1 ,T_2)$ to be the collection of functions such that $\bm{u}\in L^{p}((T_1 ,T_2);[H^{2s, p}(\mathbb{R}^{d})]^{d} )$, $\partial_{t} \bm{u}\in L^{p}((T_1 ,T_2)\times\mathbb{R}^{d})$, and
$$
\|\bm{u}\|_{\mathbb{H}^{(2s, p); 1}(T_1 ,T_2)}:=\left(\int_{T_1}^{T_2}\|\bm{u}(t,\cdot)\|^{p}_{[H^{2s, p}(\mathbb{R}^{d})]^d } \ dt\right)^{1/p} + \|\partial_{t}\bm{u}\|_{L^{p}((T_1 ,T_2)\times\mathbb{R}^{d})}.
$$
We write $\bm{u}\in \mathbb{H}_0^{(2s, p); 1}(T_1 ,T_2)$ if there exists a sequence of functions $\{\bm{u}_{n}\}$ such that $\bm{u}_{n}\in [C^{\infty}([T_1,T_2]\times\mathbb{R}^{d})]^d$ with $\bm{u}_{n}(T_1,\bm{x})=0$ for large $|\bm{x}|$, and
\[
\|\bm{u}_{n} - \bm{u}\|_{\mathbb{H}^{(2s, p); 1}(T_1 ,T_2)} \to 0 \quad \text{as $n\to \infty$}. 
\]
Moreover, for a domain $\Omega\subset\mathbb{R}^{d}$ and $\bm{u}$ defined on $(T_1 ,T_2)\times\Omega$ we write $\bm{u}\in \mathbb{H}^{(2s,p);1}_{0}((T_1 ,T_2)\times\Omega)$ if there exists an extension of $\bm{u}$ to $(T_1 ,T_2)\times\mathbb{R}^{d},$ i.e.,
$$
\|\bm{u}\|_{\mathbb{H}^{(2s,p);1}_{0}((T_1 ,T_2)\times\Omega)}:=\inf\{\|\overline{\bm{u}}\|_{\mathbb{H}^{(2s,p);1}((T_1 ,T_2))}:\overline{\bm{u}}\in\mathbb{H}^{(2s,p);1}_{0}(T_1 ,T_2) \text{ and } \overline{\bm{u}}|_{(T_1,T_2)\times\Omega}=\bm{u}\}<\infty.
$$
Furthermore, we take $\psi(\bm{x})=1/(1+|\bm{x}|^{d+2s})$ to introduce the weighted space $L^{1}(\mathbb{R}^{d},\psi)$ and denote 
\begin{equation}
\label{eq:2.1}
\|\bm{u}\|_{L^{p}((0,T);L^{1}(\mathbb{R}^{d},\psi))}:=\|\psi \bm{u}\|_{L^{p}((0,T);L^{1}(\mathbb{R}^{d}))},
\end{equation}

For $T\in(0,\infty)$ we denote $(0,T)\times\mathbb{R}^{d}:=\mathbb{R}_{T}^{d}$,  and write  $\mathbb{H}^{(2s, p); 1}(T):=\mathbb{H}^{(2s, p); 1}(0,T)$ and $\mathbb{H}_{0}^{(2s, p);1}(T):=\mathbb{H}_{0}^{(2s, p);1}(0,T)$. We use the notation $\bm{u}\in\mathbb{H}_{0, loc}^{(2s, p);1}(\mathbb{R}_{T}^{d})$ to indicate a function satisfying $\bm{u}\in\mathbb{H}_{0}^{(2s, p);1}((0,T)\times B_{R})$ for all $R>0$.

We are now ready to state an auxiliary result that will be used to prove the continuity of $\mathbb{L}$ as well as {\em a priori} estimates. 
\begin{theorem}
\label{timedep}
Let $s \in (0,1)$, $\lambda \geq 0$,  $p \in (1,\infty)$, $T \in (0, \infty)$ and $\bm{g} \in [L^p(\mathbb{R}_{T}^d)]^d$.  Suppose that the kernel $K$ satisfies \eqref{Ellipticity} and \eqref{cancellation} and that $\mathbb{L}$ is as defined  in \eqref{leading-operator} for smooth functions. Then $\partial_{t} + \mathbb{L}$ has a unique continuous extension from $\mathbb{H}_{0}^{(2s, p);1}(0,T)$ to $[L^{p}(\mathbb{R}^{d}_{T})]^{d}$. Moreover,  there exists $N>0$, depending only on $d$, $s$, $p$, $\alpha_1$ and $\alpha_2$,  such that 
\begin{enumerate}
\item[(A)] 
the fractional parabolic system of equations 
\begin{equation}\label{eqtimedep}
\partial_t \bm{u} + (-\Delta)^s \bm{u} + \lambda \bm{u}  = \bm{g} \quad \text{in } \mathbb{R}_{T}^d 
\end{equation}
has a unique solution $\bm{u}\in \mathbb{H}_{0}^{(2s, p);1}(0,T)$ and 
\begin{equation}
\label{esttimedep}
\|\partial_t  \bm{u}\|_{[L^p(\mathbb{R}_{T}^d)]^d} +  \|\mathbb{L} \bm{u}\|_{[L^p(\mathbb{R}_{T}^d)]^d} + \lambda \| \bm{u}\|_{[L^p(\mathbb{R}_{T}^d)]^d}   \leq N \|\bm{g}\|_{[L^p(\mathbb{R}_{T}^d)]^d}.  
\end{equation}
\item[(B)] Moreover, the strongly-coupled fractional parabolic system of equations 
\begin{equation}\label{eqtimedep-apriori}
\partial_t \bm{u} + \mathbb{L} \bm{u} + \lambda \bm{u}  = \bm{g} \quad \text{in } \mathbb{R}_{T}^d, 
\end{equation}
has a unique solution $\bm{u}\in  \mathbb{H}_{0}^{(2s, p);1}(0,T)$ and  
\begin{equation}
\label{esttimedep-apriori}
\|\partial_t  \bm{u}\|_{[L^p(\mathbb{R}_{T}^d)]^d} +  \|(-\Delta)^s \bm{u} \|_{[L^p(\mathbb{R}_{T}^d)]^d} + \lambda \| \bm{u}\|_{[L^p(\mathbb{R}_{T}^d)]^d}   \leq N \|\bm{g}\|_{[L^p(\mathbb{R}_{T}^d)]^d}. 
\end{equation}
\end{enumerate}
\end{theorem}
We remark that for a given $\lambda > 0$, by considering \eqref{eqtimedep} component wise \cite[Theorem 2.6]{Dong:2023aa} guarantees the existence of a unique solution $\bm{u} \in \mathbb{H}_{0}^{(2s, p);1}(0,T) $ accompanied by the estimate
\begin{equation}
\label{basicone}
\|\partial_t  \bm{u}\|_{[L^p(\mathbb{R}_{T}^d)]^d} +  \|\mathcal{L} \bm{u}\|_{[L^p(\mathbb{R}_{T}^d)]^d} + \lambda \| \bm{u}\|_{[L^p(\mathbb{R}_{T}^d)]^d}   \leq N \|\bm{g}\|_{[L^p(\mathbb{R}_{T}^d)]^d}, 
\end{equation}
where $\mathcal{L}$ is as defined in \eqref{scaler-main-op}. 
However, it is not clear how to obtain \eqref{esttimedep} from the component-wise estimate. This is partly because $\mathbb{L}$ is a strongly-coupled operator and no direct comparison between $\|\mathcal{L}u\|_{L^{p}(\mathbb{R}_{T}^d)}$ and $\|\mathbb{L} \bm{u}\|_{[L^p(\mathbb{R}_{T}^d)]^d}$ is available yet for $p\neq 2$.  Thus, the point of part (A) of Theorem \ref{timedep} is to establish the validity of \eqref{basicone} proved in \cite[Theorem 2.6]{Dong:2023aa} when $\mathcal{L}$ is substituted by the strongly-coupled operator $\mathbb{L}$. The estimate \eqref{esttimedep} will be proved following the approach used in \cite{{Dong:2023aa}} to prove \eqref{basicone}.  We will demonstrate below that part (A) of Theorem \ref{timedep} is sufficient to prove that  $\mathbb{L}$ has a unique continuous extension from  $[H^{2s,p}(\mathbb{R}^d )]^d$ to $[L^p(\mathbb{R}^d )]^d$. Again, following the approach used in \cite{{Dong:2023aa}}, one can use the continuity of $\mathbb{L}$ to establish the existence of a unique solution (corresponding to $\lambda > 0$) to the strongly-coupled fractional parabolic system of equation \eqref{eqtimedep-apriori}  as well as the associated estimate \eqref{esttimedep-apriori}.

Once Theorem \ref{timedep} is established, the continuity of $\mathbb{L}$ and the necessary {\em a priori } estimates for applying the method of continuity immediately follow as a corollary, see \cite[Remark 2.7]{Dong:2023aa}. 
\begin{corollary} \label{Lbdd}
Let $s \in (0,1)$, $\lambda \geq 0$,  $p \in (1,\infty)$. Suppose that the kernel $K$ satisfies \eqref{Ellipticity} and \eqref{cancellation} and that $\mathbb{L}$ is as defined  in \eqref{leading-operator} for smooth functions.
 Then there exists $N>0$, depending only on $d$, $s$, $p$, $\alpha_1$ and $\alpha_2$,  such that  
 \begin{equation}\label{contL}
  \left\|\mathbb{L} \bm{v} \right\|_{[L^p (\mathbb{R}^d )]^d}  \leq N \left\|(-\Delta)^s \bm{v}\right\|_{[L^p (\mathbb{R}^d)]^d},\quad \forall \, \bm{v}\in [H^{2s,p}(\mathbb{R}^d )]^d. 
 \end{equation}
  Moreover, for any $\lambda>0$ and any $ \bm{v}\in [H^{2s,p}(\mathbb{R}^d )]^d$, 
 \begin{equation}\label{aprioriL}
 \begin{split}
 \|(-\Delta)^s \bm{v}\|_{[L^p(\mathbb{R}^d)]^d} + \lambda \|\bm{v}\|_{[L^p(\mathbb{R}^d)]^d} &\leq N \|\mathbb{L}\bm{v} + \lambda \bm{v}\|_{[L^p(\mathbb{R}^d)]^d}\\
  \|\bm{v}\|_{[H^{2s, p}(\mathbb{R}^{d})]^{d}} &\leq  N\max\{1, \lambda^{-1}\} \| \mathbb{L}\bm{v} + \lambda \bm{v}\|_{[L^{p}(\mathbb{R}^{d})]^{d}}. 
 \end{split}
 \end{equation}
\end{corollary}
\begin{proof} 
Let $\bm{v} \in [H^{2s,p}(\mathbb{R}^d )]^d $.    
Let $\eta \in C^\infty_0(\mathbb{R})$ such that $\eta(0)=0$. For any given $T \in (0, \infty)$ 
\[
\bm{u}_T(t, \bm{x}) := \eta\left(t/T\right) \bm{v}(\bm{x}), \quad (t, \bm{x} ) \in \mathbb{R}^d_{T}
\]
Then it is clear that $\bm{u}_T \in \mathbb{H}^{(2s,p);1}_0(0, T)$   and 
\begin{equation}\label{cont-T-version}
\partial_t \bm{u}_T + (-\Delta)^s \bm{u}_T = \frac{1}{T}\eta'\left(t/T\right) \bm{v} + \eta\left(t/T\right) (-\Delta)^s \bm{v}  \quad \text{in}  \ \mathbb{R}^d_{T} \quad \text{for all } T \in (0, \infty). 
\end{equation}
Also, for any $\lambda >0$, 
\begin{equation}\label{apriori-T-version}
\partial_t \bm{u}_T + \mathbb{L} \bm{u}_T  + \lambda \bm{u}_{T}= \frac{1}{T}\eta'\left(t/T\right) \bm{v} + \eta\left(t/T\right) (\mathbb{L}  \bm{v}  + \lambda \bm{v}) \quad \text{in}  \ \mathbb{R}^d_{T} \quad \text{for all } T \in (0, \infty). 
\end{equation}
Applying the {\em a priori estimate} \eqref{esttimedep} in part (A) of Theorem \ref{timedep} for $\lambda =0$ to \eqref{cont-T-version}, we have for all $T \in (0, \infty)$ that
\begin{align*}
\|\mathbb{L} \bm{u}_T\|_{[L^p (\mathbb{R}_T^d)]^d} &\leq N \left\|\frac{1}{T}\eta'\left(t/T\right) \bm{v} + \eta\left(t/T\right) (-\Delta)^s \bm{v}\right\|_{[L^p (\mathbb{R}_T^d)]^d}\\
&\leq N \left\|\frac{1}{T}\eta'\left(t/T\right) \bm{v} \right\|_{[L^p (\mathbb{R}_T^d)]^d} +N \left\| \eta\left(t/T\right) (-\Delta)^s \bm{v}\right\|_{[L^p (\mathbb{R}_T^d)]^d}\\
&= N \left\|\frac{1}{T}\eta'\left(t/T\right)  \right\|_{L^p (0,T)} \left\|  \bm{v} \right\|_{[L^p (\mathbb{R}^d)]^d} +N \left\| \eta\left(t/T\right)\right\|_{L^p (0,T)} \left\|(-\Delta)^s \bm{v}\right\|_{[L^p (\mathbb{R}^d)]^d}.
\end{align*}
Using the fact that 
\begin{align*}
\|\mathbb{L} \bm{u}_T\|_{[L^p (\mathbb{R}_T^d)]^d} &= \left\|\eta\left(t/T\right) \mathbb{L} \bm{v}(\bm{x}) \right\|_{[L^p (\mathbb{R}_T^d)]^d}
 = \left\|\eta\left(t/T\right) \right\|_{L^p (0,T)} \left\|\mathbb{L} \bm{v}(\bm{x}) \right\|_{[L^p (\mathbb{R}^d )]^d},
\end{align*}
 we get 
\begin{align*}
 \left\|\mathbb{L} \bm{v}(\bm{x}) \right\|_{[L^p (\mathbb{R}^d )]^d} &\leq  N \frac{\left\|\frac{1}{T}\eta'\left(t/T\right)  \right\|_{L^p (0,T)}}{\left\| \eta\left(t/T\right)\right\|_{L^p (0,T)}} \left\|  \bm{v} \right\|_{[L^p (\mathbb{R}^d)]^d} +N  \left\|(-\Delta)^s \bm{v}\right\|_{[L^p (\mathbb{R}^d)]^d},
\end{align*}
Moreover, 
\[
\frac{\left\|\frac{1}{T}\eta'\left(t/T\right)  \right\|_{L^p (0,T)}}{\left\| \eta\left(t/T\right)\right\|_{L^p (0,T)}} =  \frac{ \frac{1}{T^{1+1/p}} \left\|\eta'\right\|_{L^p (0,1)}}{  \frac{1}{T^{1/p}} \left\| \eta\right\|_{L^p (0,1)}} \longrightarrow 0 \quad \text{as} \ T \to \infty.
\]
Estimate \eqref{aprioriL} also follows similarly by applying the {\em a priori estimate} in part (B) of Theorem \ref{timedep} to \eqref{apriori-T-version}. 
\end{proof}
We can now apply \eqref{apriori-T-version} to $\mathbb{L}_{\tau}$, see the previous section,  to obtain the {\em a priori} estimate and \eqref{cont-T-version} to complete the application of the method of continuity to prove Theorem \ref{mainthr}. 
\begin{remark}\label{Lbdd-time}
As a simple consequence of Corollary \ref{Lbdd}, we have the following results. Let $s \in (0,1)$, $\lambda \geq 0$,  $p \in (1,\infty)$. Suppose that the kernel $K$ satisfies \eqref{Ellipticity} and \eqref{cancellation}. Then 
\begin{enumerate}
\item for any $u\in [\mathbb{H}^{2s, p }(\mathbb{R}^d)]^d$
\[
\left\|(-\mathring{\Delta})^s \bm{v} \right\|_{[L^p (\mathbb{R}^d )]^d}\asymp\left\|\mathbb{L} \bm{v} \right\|_{[L^p (\mathbb{R}^d )]^d}  \asymp\left\|(-\Delta)^s \bm{v}\right\|_{[L^p (\mathbb{R}^d)]^d}\asymp \left\|\mathcal{L} \bm{v} \right\|_{[L^p (\mathbb{R}^d )]^d}. 
\]
This should be compared with the fractional Korn's inequality proved in \cite{Scott-Mengesha-Korn}. 
\item  Similarly, for any $\bm{u} \in \mathbb{H}^{(2s, p); 1}_{0}(\mathbb{R}^{d}_{T})$ and almost every $t\in (0, T)$, we have 
\[
\left\|\mathbb{L} \bm{u}(t, \cdot) \right\|_{[L^p (\mathbb{R}^d )]^d}  \leq N \left\|(-\Delta)^s \bm{u}(t, \cdot)\right\|_{[L^p (\mathbb{R}^d)]^d}. 
\]
Integrating in $t$ we obtain that 
\[
\left\|\mathbb{L} \bm{u} \right\|_{[L^p (\mathbb{R}_{T}^d )]^d}  \leq N \left\|(-\Delta)^s \bm{u}\right\|_{[L^p (\mathbb{R}_{T}^d)]^d}.
\]
\end{enumerate}
\end{remark}
\section{Estimates for solutions of fractional parabolic systems}\label{System-Parabolic}
In this section, we prove Theorem \ref{timedep} following the framework established in \cite{Dong:2023aa}. For maximum precision and ease of reference, we will directly quote necessary results from  \cite{Dong:2023aa} while accepting the inherent redundancy.  The proof of the {\em a priori} estimate for the solution to the system of parabolic fractional equations uses a level set argument to estimate the “layer-cake” representation of $\|\mathbb{L}\bm{u}\|_{[L^{p}]^d}$.  The proof proceeds via an iterative argument. The theorem's validity, which is known for $p=2$,  is extended to larger values of $p$ incrementally via careful decomposition and estimation of the solution. The critical feature is that the increment, which is possible by Sobolev embedding,
  is uniform, independent of the previous $p$ value. The uniformity ensures that the argument can be applied repeatedly to prove the theorem for a much larger range of $p$. 
  
  The proofs for parts (A) and (B) of Theorem \ref{timedep} are analogous in structure. Most importantly, unlike the scalar result in \cite{Dong:2023aa} where the arguments for $(-\Delta)^{s}$ and $\mathcal{L}$ were interchangeable, we cannot combine the proofs here. This is due to the fundamental difference between the strongly-coupled operator $\mathbb{L}$ and and the uncoupled operator $(-\Delta)^{s}$. Given the parallelism of the two arguments, we will detail the proof of part (A) of Theorem \ref{timedep},  while proving part (B) only in the case $p=2.$ 
  \subsection{The case for $p=2$}
  For part (A),  the existence and uniqueness of the solution of the equation
\[
\partial_t \bm{u} + (-\Delta)^s \bm{u} + \lambda \bm{u}  = \bm{f} \quad \text{in } \mathbb{R}^d \times (0, T)
\]
follows directly by \cite[Theorem 2.6]{Dong:2023aa} , applied component-wise. The estimate  $\eqref{esttimedep}$ in the case $p=2$, is a simple consequence of $\eqref{basicone}$ and Fourier transform. Indeed,  
\begin{align*}
\|\mathbb{L} \bm{u} \|_{[L^2((0, T) \times \mathbb{R}^d )]^d}^{2} =  \int_{0}^{T}\|\mathbb{L} \bm{u}(t, \cdot) \|^{2}_{[L^2(\mathbb{R}^d)]^d} \, dt
&=  \int_{0}^{T} \|\widehat{\mathbb{L} \bm{u}}(t, \cdot) \|^{2}_{[L^2(\mathbb{R}^d)]^d} \,dt \\
&= \int_{0}^{T} \|\mathbb{M}(\cdot) \widehat{\bm{u}}(t, \cdot) \|^2_{[L^2(\mathbb{R}^d)]^d} \,dt \\
&\leq C \int_{0}^{T}\||\cdot|^{2s} \widehat{\bm{u}}(t, \cdot)\|^2_{[L^2(\mathbb{R}^d)]^d} \, dt \\
&= C \|(-\Delta)^s \bm{u} \|^{2}_{[L^2((0, T) \times  \mathbb{R}^d)]^d}.
\end{align*}
For part (B), we use the method of continuity to prove existence of a  solution to \eqref{eqtimedep-apriori}. To that end, first note that the above calculation shows that  $\partial_{t} + \mathbb{L} + \lambda \mathbb{I} : \mathbb{H}_0^{(2s,2); 1}(T) \to [L^{2}(\mathbb{R}_{T}^{d})]^d$ is a continuous map. Next we obtain an  {\em a priori} estimate, which is given in the following lemma.
\begin{lemma}
    There exists a constant $N>0$ such that for any $\bm{u}\in \mathbb{H}_0^{(2s,2); 1}(T)$ and any $\lambda\geq 0$
   \begin{align*}
\|\partial_t  \bm{u}\|_{[L^2(\mathbb{R}_{T}^d)]^d} +  \|(-\Delta)^s \bm{u} \|_{[L^2(\mathbb{R}_{T}^d)]^d} + \lambda \| \bm{u}\|_{[L^2(\mathbb{R}_{T}^d)]^d} \leq N\|\partial_t \bm{u} + \mathbb{L} \bm{u} + \lambda \bm{u}\|_{[L^{p}(\mathbb{R}^{d}_{T})]^d}. 
\end{align*}
\end{lemma}
\begin{proof}
Suppose that $\bm{u}\in \mathbb{H}_0^{(2s,2); 1}(T)$. Set $\bm{g}:=\partial_t \bm{u} + \mathbb{L} \bm{u} + \lambda \bm{u}.$  
Without loss of generality, using the continuity of the operator and the density of smooth functions in $\mathbb{H}^{(2s,2);1}_0 (0, T)$,  we assume that $\bm{u} \in [C^\infty_0(\mathbb{R}_{T}^d)]^d$ with $\bm{u}(0, \cdot)=0$. 
Taking the dot product with $(-\Delta)^s \bm{u}$ in both sides of \eqref{eqtimedep-apriori} and integrating over $\mathbb{R}_T^d$, we get 
\begin{equation}\label{multint}
\int_{\mathbb{R}_T^d} \partial_t \bm{u} \cdot (-\Delta)^s \bm{u} \, d \bm{x}  + \int_{\mathbb{R}_T^d}  \mathbb{L} \bm{u} \cdot  (-\Delta)^s \bm{u} \, d\bm{x}+ \lambda \int_{\mathbb{R}_T^d}  \bm{u} \cdot  (-\Delta)^s \bm{u} \, d \bm{x} = \int_{\mathbb{R}_T^d}  \bm{g} \cdot (-\Delta)^s \bm{u} \, d\bm{x}. 
\end{equation}
For the first term on the left-hand side of \eqref{multint}, by applying Plancherel's formula and by recalling that $\widehat{(-\Delta)^s \bm{u}}( t, \bm{\xi}) = |\bm{\xi}|^{2s} \widehat{\bm{u}} (t, \bm{\xi}) $, 

\begin{align*}
\int_{\mathbb{R}_T^d} &\partial_t \bm{u}(t, \bm{x}) \cdot  (-\Delta)^s \bm{u}(t, \bm{x}) \ d\bm{x} \ dt\\  &= \int_{\mathbb{R}_T^d}  \widehat{(-\Delta)^s \bm{u}}(t, \bm{\xi}) \cdot \overline{\widehat{\partial_t \bm{u}}(t, \bm{\xi})}  \ d\bm{\xi} \ dt
=  \int_{\mathbb{R}_T^d} |\bm{\xi}|^{2s} \widehat{ \bm{u}}(t, \bm{\xi}) \cdot \overline{\partial_t  \widehat{\bm{u}}(t, \bm{\xi})}  \ d\bm{\xi} \ dt\\
&=  \int_{\mathbb{R}_T^d} |\bm{\xi}|^{2s} \Re[ \widehat{ \bm{u}}](t,\bm{\xi}) \cdot \partial_t  \Re[ \widehat{ \bm{u}}](t,\bm{\xi})  \ d\bm{\xi} \ dt +   \int_{\mathbb{R}_T^d} |\bm{\xi}|^{2s} \Im[ \widehat{ \bm{u}}](t,\bm{\xi}) \cdot \partial_t  \Im[ \widehat{ \bm{u}}](t,\bm{\xi})  \ d\bm{\xi} \ dt\\
&+ \imath \int_{\mathbb{R}_T^d} |\bm{\xi}|^{2s} \left(\Im[ \widehat{ \bm{u}}](t,\bm{\xi}) \cdot \partial_t  \Re[ \widehat{ \bm{u}}](t,\bm{\xi})  + \Re[ \widehat{ \bm{u}}](t,\bm{\xi})\cdot  \partial_t  \Im[ \widehat{ \bm{u}}](t,\bm{\xi})\right) \ d\bm{\xi} \ dt. 
\end{align*}
Using the fact that $\int_{\mathbb{R}_T^d}   \partial_t \bm{u}( t, \bm{x})  \cdot (-\Delta)^s \bm{u}( t, \bm{x}) \ d\bm{x} \ dt$ is real, we conclude that 
\[
\begin{split}
\int_{\mathbb{R}_T^d} &\partial_t \bm{u}(t, \bm{x}) \cdot  (-\Delta)^s \bm{u}(t, \bm{x}) \ d\bm{x} \ dt\\
&=  \int_{\mathbb{R}_T^d} |\bm{\xi}|^{2s} \Re[ \widehat{ \bm{u}}](t,\bm{\xi}) \cdot \partial_t  \Re[ \widehat{ \bm{u}}](t,\bm{\xi})  \ d\bm{\xi} \ dt+  \int_{\mathbb{R}_T^d} |\bm{\xi}|^{2s} \Im[ \widehat{ \bm{u}}](t,\bm{\xi}) \cdot \partial_t  \Im[ \widehat{ \bm{u}}](t,\bm{\xi})  \ d\bm{\xi} \ dt.
\end{split}
\]
Using chain rule for $\partial_t$, we have  
\[
 \Re[ \widehat{ \bm{u}}](t,\bm{\xi}) \cdot \partial_t  \Re[ \widehat{ \bm{u}}](t,\bm{\xi}) +  \Im[ \widehat{ \bm{u}}](t,\bm{\xi}) \cdot \partial_t  \Im[ \widehat{ \bm{u}}](t,\bm{\xi}) = \frac{1}{2} \partial_t |  \Re[ \widehat{ \bm{u}}]|^2(t,\bm{\xi}) + \frac{1}{2} \partial_t |  \Im[ \widehat{ \bm{u}}]|^2(t,\bm{\xi}) = \frac{1}{2} \partial_t | \widehat{ \bm{u}}|^2(t,\bm{\xi}).
\]
Thus, using fundamental theorem of calculus with $\bm{u}(0, \cdot)=0$ in the last line, 
\begin{align*}
\int_{\mathbb{R}_T^d} \partial_t \bm{u}(t, \bm{x}) \cdot  (-\Delta)^s \bm{u}(t, \bm{x}) \ d\bm{x} \ dt  &=  \int_{\mathbb{R}_T^d} |\bm{\xi}|^{2s}  \frac{1}{2} \partial_t | \widehat{ \bm{u}}|^2(t,\bm{\xi})
 \ d\bm{\xi} \ dt\\
 &=  \frac{1}{2} \int_{\mathbb{R}^d} |\bm{\xi}|^{2s}  \int_0^T \partial_t | \widehat{ \bm{u}}|^2(t,\bm{\xi})
 \ dt \ d\bm{\xi} \\
 &=  \frac{1}{2} \int_{\mathbb{R}^d} |\bm{\xi}|^{2s}    | \widehat{ \bm{u}}|^2(T,\bm{\xi}) \ d\bm{\xi} \geq 0.
\end{align*}
For the second term on the left-hand side of \eqref{multint}, by recalling that $\widehat{\mathbb{L} \bm{u}}(t,\bm{\xi}) = \mathbb{M}(\bm{\xi}) \widehat{\bm{u}} (t,\bm{\xi}) $  where $\mathbb{M}(\bm{\xi}) = \int_{\mathbb{R}^d} 
\frac{{\bm y}\otimes {\bm y}}{|{\bm y}|^2}\left(e^{\imath 2\pi{\bm \xi}\cdot {\bm y}} - 1 -2\pi \imath{\bm  \xi}\cdot {\bm y} \chi^{(s)}(\bm{y}) \right)K(\bm{y}) d\bm{y}$,
\begin{align*}
\int_{\mathbb{R}_T^d} \mathbb{L} \bm{u}(t, \bm{x}) \cdot  (-\Delta)^s \bm{u}(t, \bm{x}) \ d\bm{x} \ dt  &= \int_{\mathbb{R}_T^d}   \widehat{(-\Delta)^s \bm{u}}(t,\bm{\xi}) \cdot \overline{\widehat{\mathbb{L} \bm{u}}(t,\bm{\xi})}  \ d\bm{\xi} \ dt\\
&=  \int_{\mathbb{R}_T^d} |\bm{\xi}|^{2s} \widehat{ \bm{u}}(t,\bm{\xi})^{T} \overline{\mathbb{M}(\bm{\xi})  \widehat{\bm{u}}(t,\bm{\xi})}  \ d\bm{\xi} \ dt\\
&=  \int_{\mathbb{R}_T^d} |\bm{\xi}|^{2s} \widehat{ \bm{u}}(t,\bm{\xi})^{T} \mathbb{M}^e(\bm{\xi})  \overline{\widehat{\bm{u}}(t,\bm{\xi})}  \ d\bm{\xi} \ dt, 
\end{align*}
where
\begin{align*}
\mathbbm{M}^e(\bm{\xi}) &:= \int_{\mathbb{R}^d} \left(\frac{\bm{y} \otimes \bm{y}}{|\bm{y}|^2 }\right) \left( e^{2\pi i \bm{y} \cdot \bm{\xi} } - 1 - 2\pi i \bm{\xi} \cdot \bm{y} \chi^{(s)}(\bm{y}) \right) \ K^e(\bm{y}) \ d\bm{y}\\
&=  \int_{\mathbb{R}^d} \left(\frac{\bm{y} \otimes \bm{y}}{|\bm{y}|^2 }\right) ( 1 - \cos(2\pi \bm{y} \cdot \bm{\xi} ))  \ K^e(\bm{y}) \ d\bm{y} 
\end{align*}
is the even part of $\mathbb{M}$. Now, for every $\bm{\xi}\in \mathbb{R}^d$, the matrix $\mathbbm{M}^e(\bm{\xi})$ is symmetric with real entries, hence the least of its eigenvalues is given by $\min_{\bm{\mu} \in \mathbb{S}^{d-1}} \bm{\mu}^T \mathbbm{M}^e(\bm{\xi}) \bm{\mu}$. We estimate it from below as a function of $\bm{\xi}$ as,
\begin{align*}
\min_{\bm{\mu} \in \mathbb{S}^{d-1}} \bm{\mu}^T \mathbbm{M}^e(\bm{\xi}) \bm{\mu} &= \min_{\bm{\mu} \in \mathbb{S}^{d-1}} \int_{\mathbb{R}^d}   ( 1 - \cos(2\pi \bm{y} \cdot \bm{\xi} ))  \left| \bm{\mu} \cdot \frac{\bm{y}}{|\bm{y}|} \right|^2 \ K^e(\bm{y}) \ d\bm{y}\\
&\geq  \min_{\bm{\mu} \in \mathbb{S}^{d-1}}  \alpha_1 \int_{\mathbb{R}^d}   \frac{( 1 - \cos(2\pi \bm{y} \cdot \bm{\xi} )) } {|\bm{y}|^{d+2s}} \left| \bm{\mu} \cdot \frac{\bm{y}}{|\bm{y}|} \right|^2  \ d\bm{y}\\
&= \min_{\bm{\mu} \in \mathbb{S}^{d-1}}  \alpha_1  (2\pi |\bm{\xi}|)^{2s} \int_{\mathbb{R}^d}   \frac{( 1 - \cos(2\pi  \bm{h} \cdot \frac{\bm{\xi}}{|\bm{\xi}|} )}{|\bm{h}|^{d+2s}} \left| \bm{\mu} \cdot \frac{\bm{h}}{|\bm{h}|} \right|^2  \ d\bm{h}\\
& \geq C(2\pi |\bm{\xi}|)^{2s}, 
\end{align*} 
where we have made the substitution $\bm{h} = 2\pi |\bm{\xi}| \bm{y}$ and used the fact that 
 \[C  = \alpha_1  \min_{\bm{\mu}\in \mathbb{S}^{d-1}, \bm{\xi }\in \mathbb{R}^{d}\setminus\{0\} } \int_{\mathbb{R}^d}   \frac{( 1 - \cos(2\pi  \bm{h} \cdot \frac{\bm{\xi}}{|\bm{\xi}|} )}{|\bm{h}|^{d+2s}} \left| \bm{\mu} \cdot \frac{\bm{h}}{|\bm{h}|} \right|^2  \ d\bm{h} > 0.  \]
The latter follows from the observation that the nonnegative function $\bm{\Psi}(\bm{\nu}, \bm{\mu}) : \mathbb{S}^{d-1} \times \mathbb{S}^{d-1} \to \mathbb{R}$ given by 
\[
 \bm{\Psi}(\bm{\nu}, \bm{\mu}) := \int_{\mathbb{R}^d}   \frac{1 - \cos(2\pi  \bm{h} \cdot \bm{\nu} )}{|\bm{h}|^{d+2s}} \left| \bm{\mu} \cdot \frac{\bm{h}}{|\bm{h}|} \right|^2  \ d\bm{h} 
\]
is strictly positive for any $\bm{\nu}, \bm{\mu} \in \mathbb{S}^{d-1} $ since $\frac{1 - \cos(2\pi  \bm{h} \cdot \bm{\nu} )}{|\bm{h}|^{d+2s}} \left| \bm{\mu} \cdot \frac{\bm{h}}{|\bm{h}|} \right|^2 >0  $ almost everywhere and that $\bm{\Psi}$ is clearly continuous on the compact set $\mathbb{S}^{d-1} \times \mathbb{S}^{d-1}$, so its minimum  $\min_{\bm{\nu}, \bm{\mu} \in \mathbb{S}^{d-1}} \bm{\Psi}(\bm{\nu}, \bm{\mu}) = C>0.$
We conclude that 
\begin{align}
\int_{\mathbb{R}_T^d} \mathbb{L} \bm{u}(t, \bm{x}) \cdot (-\Delta)^s \bm{u}(t, \bm{x}) \ d\bm{x} \ dt  &=  \int_{\mathbb{R}_T^d} |\bm{\xi}|^{2s} \widehat{ \bm{u}}(t,\bm{\xi})   \mathbb{M}^e(\bm{\xi})  \overline{\widehat{\bm{u}}(t,\bm{\xi})}  \ d\bm{\xi} \ dt \nonumber\\
&\geq  \int_{\mathbb{R}_T^d} |\bm{\xi}|^{2s}   |\widehat{\bm{u}}(t,\bm{\xi} )|^{2}  \min_{\bm{v} \in \mathbb{S}^{d-1}} \bm{v}^T {\mathbbm{M}^e}(\bm{\xi}) \bm{v} \ d\bm{\xi} \ dt \nonumber\\
 &\geq C   \int_{\mathbb{R}_T^d} |\bm{\xi}|^{4s} |\widehat{\bm{u}}(t,\bm{\xi})|^2   \ d\bm{\xi}  \ dt \nonumber\\
 &= C \| (-\Delta)^s \bm{u}\|^2_{[L^2(\mathbb{R}_T^d)]^d}.  \label{bdlapl}
\end{align}
Similarly, taking the dot product with $\bm{u}$ in both sides of \eqref{eqtimedep-apriori} and integrating over $\mathbb{R}_T^d$, we get 
\begin{equation} \label{multint2}
 \int_{\mathbb{R}_T^d} \partial_t \bm{u} \cdot \bm{u} \ d\bm{x} \ dt   +  \int_{\mathbb{R}_T^d}  \mathbb{L} \bm{u} \cdot \bm{u}\ d\bm{x} \ dt  + \lambda \int_{\mathbb{R}_T^d} | \bm{u}|^2 \ d\bm{x} \ dt  =  \int_{\mathbb{R}_T^d}  \bm{g}  \cdot \bm{u}\ d\bm{x} \ dt 
\end{equation}
and therefore 
{\begin{align}\label{drop1}
\int_{\mathbb{R}_T^d} \partial_t \bm{u}(t, \bm{x})   \cdot \bm{u}(t, \bm{x}) \ d\bm{x} \ dt  &=   \frac{1}{2} \int_{\mathbb{R}^d}     | \widehat{ \bm{u}}|^2(T, \bm{\xi}) \ d\bm{\xi} \geq 0, 
\end{align}}
\begin{align} \label{drop2}
\int_{\mathbb{R}_T^d} \mathbb{L} \bm{u}(t, \bm{x}) \cdot  \bm{u}(t, \bm{x}) \ d\bm{x} \ dt  
 &\geq C   \int_{\mathbb{R}_T^d} |\bm{\xi}|^{2s} |\widehat{\bm{u}}(t,\bm{\xi})|^2   \ d\bm{\xi}  \ dt \geq 0. 
\end{align}
It then also follows by H\"older's inequality that,
\[
\lambda \int_{\mathbb{R}_T^d} | \bm{u}|^2 \ d\bm{x} \ dt  \leq  \int_{\mathbb{R}_T^d}  \bm{g}  \cdot \bm{u}\ d\bm{x} \ dt \leq \|\bm{g}\|_{[L^{2}(\mathbb{R}^d_T)]^d}\|\bm{u}\|_{[L^{2}(\mathbb{R}^d_T)]^d}
\]
and hence 
\begin{equation}\label{bdd}
\lambda^2 \int_{\mathbb{R}_T^d} | \bm{u}(t, \bm{x})|^2 \ d\bm{x} \ dt  \leq \int_{\mathbb{R}_T^d}  |\bm{g}(t, \bm{x})|^2   \ d\bm{x} \ dt.  
\end{equation}
Therefore, using \eqref{eqtimedep-apriori}, triangle inequality, \eqref{bdlapl}, \eqref{bdd} and the bound $\|\mathbb{L} \bm{u}\|_{[L^2(\mathbb{R}_{T}^d)]^d} \leq N \|(-\Delta)^s \bm{u} \|_{[L^2(\mathbb{R}_{T}^d)]^d}$, we conclude that
\begin{align*}
\|\partial_t  \bm{u}\|_{[L^2(\mathbb{R}_{T}^d)]^d} &+  \|(-\Delta)^s \bm{u} \|_{[L^2(\mathbb{R}_{T}^d)]^d} + \lambda \| \bm{u}\|_{[L^2(\mathbb{R}_{T}^d)]^d} \\ &\leq \|\bm{g} - \mathbb{L} \bm{u} - \lambda \bm{u} \|_{[L^2(\mathbb{R}_{T}^d)]^d} +  \|(-\Delta)^s \bm{u} \|_{[L^2(\mathbb{R}_{T}^d)]^d} + \lambda \| \bm{u}\|_{[L^2(\mathbb{R}_{T}^d)]^d} \\
&\leq \|\bm{g}\|_{[L^2(\mathbb{R}_{T}^d)]^d} + \|\mathbb{L} \bm{u}\|_{[L^2(\mathbb{R}_{T}^d)]^d} + \lambda \|\bm{u} \|_{[L^2(\mathbb{R}_{T}^d)]^d} +  \|(-\Delta)^s \bm{u} \|_{[L^2(\mathbb{R}_{T}^d)]^d} + \lambda \| \bm{u}\|_{[L^2(\mathbb{R}_{T}^d)]^d} \\
 &\leq N \|\bm{g}\|_{[L^2(\mathbb{R}_{T}^d)]^d}.
\end{align*}
\end{proof}
Finally, we show the existence of a solution to the equation corresponding to the operator $(-\mathring{\Delta})^{s}$. 
\begin{lemma}
   Suppose that $\bm{g}\in [L^{2}(\mathbb{R}_T^{d})]^d$ and $\lambda\geq 0$.  Then there exists  $\bm{u}\in \mathbb{H}^{(2s, 2); 1}_{0}(T)$ that solves  the strongly-coupled system of parabolic fractional equation
\[
\partial_t \bm{u} + \alpha_1(-\mathring{\Delta})^{s}\bm{u} + \lambda \bm{u} =\bm{g}  
\]
where the operator $(-\mathring{\Delta})^{s} $ is as defined in  \eqref{fractional-Lame} and $\alpha_1$ is the ellipticity constant in \eqref{Ellipticity}.
\end{lemma}
\begin{proof}
    By making the change of variables $\bm{v}(t,\bm{x}) = e^{\lambda t}\bm{u}(t,\bm{x}) $, it suffices to show existence of solution to 
\begin{equation}\label{base-eqn-2}
    \partial_t \bm{v} + \alpha_1(-\mathring{\Delta})^{s}\bm{v} =\bm{f}, 
    \end{equation}
    where $\bm{f} (t, \bm{x}) = e^{\lambda t} \bm{g}(t, \bm{x})$.  
    To that end, we first claim that if $\bm{v}\in [C_c(\mathbb{R}^{d}_{T})]^d$ and setting $\bm{f}:=\partial_t \bm{v} + \alpha_1(-\mathring{\Delta})^{s}\bm{v},$ then $\bm{v}$ can be represented as  
\begin{equation}\label{repr-of-v}
    \bm{v}(t, \bm{x}) = \int_{0}^{t}\int_{\mathbb{R}^{d}} \mathbb{W}(t-r, \bm{x} - \bm{y}) \bm{f}(r, \bm{y}) \, d\bm{y} \,dr
    \end{equation}
    where $\mathbb{W}(t, \bm{x})$ is the matrix heat kernel associated to the operator $\alpha_1(-\mathring{\Delta})^s$. This kernel is computed explicitly in \cite[Appendix A]{scott2022paper} with its Fourier transform for each $t$ given by 
    \[
    \widehat{\mathbb{W}}(t,\bm{\xi}) = e^{-t\alpha_1\mathbb{M}^{\Delta}(\bm{\xi})}=[e^{-\alpha_1 \ell_1 (2\pi |\bm{\xi}|)^{2s} t}] \mathbb{I}+ [e^{-(\alpha_1 \ell_1 + \ell_2) (2\pi |\bm{\xi}|)^{2s}t} - e^{-\alpha_1 \ell_1 (2\pi |\bm{\xi}|)^{2s} t}] \frac{\bm{\xi} \otimes \bm{\xi}}{|\bm{\xi}|^2}. 
    \]
  It is not difficult to see that $\mathbb{W}$ is integrable in $\mathbb{R}^{d}_{T}$  \cite{scott2022paper}. Thus,  to prove the representation \eqref{repr-of-v}, it suffices to prove that 
  \[
 \hat{\bm{v}}(t, \bm{\xi}) =\int_{0}^{t} \widehat{\mathbb{W}}( t-r, \bm{\xi}) \hat{\bm{f}}(r, \bm{\xi}) \, dr. 
  \]
  But this follows from the definition of $\bm{f} = \partial_t \bm{v} + \alpha_1(-\mathring{\Delta})^{s}\bm{v}$, taking the Fourier transform and solving the resulting system of differential equations
  \[
\partial_{t}\hat{\bm{v}} + \alpha_1\mathbb{M}^{\Delta}(\bm{\xi}) \hat{\bm{v}} = \hat{\bm{f}}, 
  \]
  where $\mathbb{M}^{\Delta}(\bm{\xi})$ is the Fourier matrix symbol given in \eqref{matrix-symbol-forfrac}.  
  We next claim the converse: if $\bm{f}\in [C_c(\mathbb{R}^{d}_T)]^d$, then $\bm{v}$ given by the representation \eqref{repr-of-v} solves the equation \eqref{base-eqn-2}, which follows from taking the Fourier transform of the $\bm{v}$ in $\bm{x}$, to get 
  \[
 \hat{\bm{v}}(t, \bm{\xi}) =\int_{0}^{t} \widehat{\mathbb{W}}( t-r, \bm{\xi}) \hat{\bm{f}}(r, \bm{\xi}) \, dr. \]
 Differentiating in time using the fundamental theorem of calculus we have 
 \[
  \partial_t \hat{\bm{v}}(t,\bm{\xi})  = -\mathbb{M}^{\Delta}(\bm{\xi}) \int_{0}^{t} \widehat{\mathbb{W}}(t-r, \bm{\xi}) \hat{\bm{f}}(r, \bm{\xi}) \, dr  + \widehat{\mathbb{W}}(0, \bm{\xi}) \hat{\bm{f}}(t, \bm{\xi}) =  -\mathbb{M}^{\Delta}(\bm{\xi}) \hat{\bm{v}}(t, \bm{\xi}) + \hat{\bm{f}}(t, \bm{\xi})
 \]
 which proves the equality $\partial_{t} \bm{v} = -\mathbb{L}\bm{v} + \bm{f}.$
  Since $\bm{f}(t, \cdot) = 0$ for small $t>0$, it follows $\bm{v}(t, \cdot) = 0$ for small $t>0$.  Moreover, $\bm{v}\in \mathbb{H}^{(2s, s); 1}(T)$ with the estimate
  \begin{equation}\label{v-estimate}
  \|\bm{v}\|_{\mathbb{H}^{(2s,2);1}(T)}\leq N\|\bm{f}\|_{[L^{2}(\mathbb{R}^{d}_{T})]^d}
  \end{equation}
  for some uniform constant $N>0$.  To establish \eqref{v-estimate}, it suffices to bound $\|\partial_t \hat{\bm{v}}\|_{[L^{2}(\mathbb{R}^{d}_{T})]^d}$ and $\|(1 + |\cdot|^{2})^s\hat{\bm{v}}\|_{[L^{2}(\mathbb{R}^{d}_{T})]^d}$. To that end, we first note the following estimates for the heat kernel which follows after a simple calculation:
  \[
\int_{0}^{T} |\mathbb{W}(t,\bm{\xi})| \, dt \leq \left\{ \begin{split} & \ 3 & \quad \text{if $|\bm{\xi}| \leq 1$}\\
 &\frac{3}{\alpha_1 \ell_1} (2\pi |\bm{\xi}|)^{-2s} &\quad \text{if $|\bm{\xi}|>1$}
\end{split}
  \right.\quad \text{and}\quad  \int_{0}^{T} |\partial_{t}\mathbb{W}(t, \bm{\xi})| \, dt \leq 3. 
  \]
  Then if $\bm{v}$ is given by \eqref{repr-of-v} we have that 
\begin{align*}
\| \bm{v} \|_{L^2([0,T]; [H^{2s,2}(\mathbb{R}^d)]^d)} &=\left( \int_0^T  \int_{\mathbb{R}^d} (1+ |\bm{\xi}|^2)^{2s} \left| \int_0^t \widehat{\mathbb{W}}(t- r,\bm{\xi}) \widehat{\bm{f}}(r, \bm{\xi}) \, dr \right|^2 \, d\bm{\xi} \, dt  \right)^{\frac{1}{2}}\\
&\leq \Biggl(  \int_{|\bm{\xi}|>1} (1+ |\bm{\xi}|^2)^{2s} 
\left(3 \frac{1 }{\alpha_1 \ell_1 (2\pi |\bm{\xi}|)^{2s} }\right)^2 \|\widehat{\bm{f}}(\cdot, \bm{\xi})\|^2_{L^2(0,T)} \, d\bm{\xi}  \\
&\quad +  \int_{|\bm{\xi}|\leq 1} (1+ |\bm{\xi}|^2)^{2s} 
(3T)^2 \|\widehat{\bm{f}}(\cdot, \bm{\xi})\|^2_{L^2(0,T)} \ d\bm{\xi}  \Biggr)^{\frac{1}{2}}\\
&\leq \Biggl( C_1  \int_{|\bm{\xi}|>1}   \|\widehat{\bm{f}}(\cdot, \bm{\xi})\|^2_{L^2(0,T)} \ d\bm{\xi}   +C_2  \int_{|\bm{\xi}|\leq 1}  \|\widehat{\bm{f}}(\cdot, \bm{\xi})\|^2_{L^2(0,T)} \ d\bm{\xi}   \Biggr)^{\frac{1}{2}}\\
&\leq C  \|\widehat{\bm{f}}\|_{[L^2(\mathbb{R}^d_T)]^d} \\
&\leq C  \|\bm{f}\|_{[L^2(\mathbb{R}^d_T)]^d}, 
\end{align*}
and 
\begin{align*}
\|\partial_t \widehat{\bm{v}}\|_{[L^2(\mathbb{R}^d_T)]^d}  
&\leq \|\bm{f}\|_{[L^2(\mathbb{R}^d_T)]^d}  + \left( \int_{\mathbb{R}^d }\int_0^T \left| \int_0^t  \partial_t \widehat{\mathbb{W}}(t-r, \bm{\xi})  \widehat{\bm{f}}(r, \bm{\xi}) \, dr \right|^2 \, dt  \, d\bm{\xi}  \right)^{\frac{1}{2}}\\
&\leq \|\bm{f}\|_{[L^2(\mathbb{R}^d_T)]^d}  + \left( \int_{\mathbb{R}^d}9 \|\widehat{\bm{f}}(\cdot, \bm{\xi})\|^2_{L^2(0,T)}
 \, d\bm{\xi}  \right)^{\frac{1}{2}}\\
 &=4 \|\bm{f}\|_{[L^2(\mathbb{R}^d_T)]^d}.
\end{align*} 
 Finally, given $\bm{f}\in [L^{2}(\mathbb{R}^{d}_{T})]^d$, let $\bm{f}_n \in [C_c(\mathbb{R}^d_T)]^d$ such that $\|\bm{f}_n - \bm{f}\|_{[L^{2}(\mathbb{R}^{d}_T)]^d} \to 0$ as $n\to \infty.$ For $n,$ let $\bm{v}_n \in \mathbb{H}^{(2s,2);1}_0(T)$ be the corresponding solution given by \eqref{repr-of-v}. Then using the estimate \eqref{v-estimate}, we see that $\{\bm{v}_n\}$ is a Cauchy sequence in $\mathbb{H}_0^{(2s,2);1}(T)$, whose unique limit satisfies \eqref{base-eqn-2}.
\end{proof}
 
  \subsection{The case for general $p$}
  A key result we need, proved in  \cite{Dong:2023aa}, is the following improved local estimate. We note that for $R>0,$ if  $\bm{u}|_{B_{R}}\in\mathbb{H}_{0}^{(2s, p);1}((0,T)\times B_{R})$, then for any $\zeta \in C^\infty_c (B_{R})$, $\zeta\bm{u}|_{B_{R}} \in \mathbb{H}_{0}^{(2s, p);1}(\mathbb{R}_{T}^{d})$. 
\begin{lemma}{\cite[Corollary 3.4]{Dong:2023aa}}
           \label{cor}
           Fix $p \in (1,\infty)$. Let $s \in (0,1),  \lambda \geq 0,  T\in (0, \infty), 0<r<R<\infty.$   Suppose that Theorem $\ref{timedep}$ holds for this $p$.   
            Assume that $ \bm{u} \in [L^p((0,T); L^1(\mathbb{R}^d , \psi))]^d$,  
            $\bm{u}|_{B_R} \in \mathbb{H}_{0}^{(2s, p);1}((0,T)\times B_{R})$  
            and satisfy $\bm{f} = \partial_t \bm{u} + (-\Delta)^s \bm{u} + \lambda \bm{u}$ in $(0, T)\times B_R.$           Let $q \in (p, \infty)$   satisfy 
        $       \frac{1}{q} = \frac{1}{p} - \frac{2s}{d+2s}.$ 

Then 
                \begin{enumerate}
                 \item if $p \leq \frac{d}{2s} + 1$, 
                 for any $\ell \in [p,q]$, 
                   \begin{equation}\label{sob-smallp}
                   \begin{split}
                   \|\bm{u}\|_{[L^\ell( (0,T) \times B_r)]^d} &\leq N \|\bm{f}\|_{[L^p((0,T) \times B_R )]^d} \\
                   &+ N \frac{\|\bm{u}\|_{[L^p((0,T) \times  B_R )]^d}}{(R-r)^{2s}} + N \frac{R^{d/p} (1+R^{d+2s})}{(R-r)^{d+2s}} \|\bm{u}\|_{[L^p((0,T); L^1(\mathbb{R}^d , \psi))]^d}
                   \end{split}
                   \end{equation}
		\item if $p > \frac{d}{2s} + 1 $, there exists $\tau  =  2s - (d+2s)/p \in (0, 1) $  such that 
		\begin{equation}\label{sob-largep}
		    \begin{split}
                   \|\bm{u}\|_{[C^{\tau/2s,\tau}( (0,T) \times B_r)]^d} &\leq N \|\bm{f}\|_{[L^p((0,T) \times B_R)]^d} \\
                   &+ N \frac{\|\bm{u}\|_{[L^p((0,T) \times B_R)]^d}}{(R-r)^{2s}} + N \frac{R^{d/p} (1+R^{d+2s})}{(R-r)^{d+2s}} \|\bm{u}\|_{[L^p((0,T); L^1(\mathbb{R}^d , \psi))]^d} 
                             \end{split}                  
                   \end{equation}
                 \end{enumerate}
                 The constant $N$ depends on $d,s, p, T, \alpha_2$. 
\end{lemma}
\begin{remark}
The proof of Lemma \ref{cor} is given in \cite{Dong:2023aa}, see the proof of \cite[Corollary 3.4]{Dong:2023aa}. 
Moreover, if $\bm{u}$  solves $\bm{f} = \partial_t \bm{u} + \mathbb{L} \bm{u} + \lambda \bm{u}$ in $(0, T)\times B_R$ instead, under the assumption of Lemma \ref{cor}, inequalities \eqref{sob-smallp} and \eqref{sob-largep} remain valid. Indeed, for  $\zeta_0 \in C^\infty_0(B_{(R+r)/2}),$ $\zeta_0 = 1$ in  $B_r$  and $ |\nabla\zeta_0| \leq 4/(R-r)$, $\zeta_{0}\bm{u}|_{B_{R}} \in \mathbb{H}_{0}^{(2s, p);1}(\mathbb{R}_{T}^{d})$ and that 
\[
           \partial_t(\zeta_0 \bm{u}) +\mathbb{L} (\zeta_0 \bm{u}) + \lambda (\zeta_0 \bm{u})= \zeta_0  \bm{f} +\mathbb{L} (\zeta_0 \bm{u})  - \zeta_0 \mathbb{L} ( \bm{u}) \quad \text{in} \  (0,T) \times \mathbb{R}^d . 
\]
applying part (B) of \Cref{timedep}, which is assumed to be true for the fixed $p$, we have 
\[
\begin{split}
\|\partial_t  (\zeta_0 \bm{u})\|_{[L^p(\mathbb{R}_{T}^d)]^d} &+  \|(-\Delta)^s (\zeta_0 \bm{u}) \|_{[L^p(\mathbb{R}_{T}^d)]^d} + \lambda \| \zeta_0 \bm{u}\|_{[L^p(\mathbb{R}_{T}^d)]^d} \\
&  \leq N \|\zeta_0  \bm{f} +\mathbb{L} (\zeta_0 \bm{u})  - \zeta_0 \mathbb{L} ( \bm{u}) \|_{[L^p(\mathbb{R}_{T}^d)]^d}\\
&\leq N \|\zeta_0  \bm{f}\|_{[L^p(\mathbb{R}_{T}^d)]^d} + N \|\mathbb{L} (\zeta_0 \bm{u})  - \zeta_0 \mathbb{L} ( \bm{u}) \|_{[L^p(\mathbb{R}_{T}^d)]^d}. 
\end{split}
\]
The commutator $\|\mathbb{L} (\zeta_0 \bm{u})  - \zeta_0 \mathbb{L} ( \bm{u}) \|_{[L^p(\mathbb{R}_{T}^d)]^d}$ can be estimated in exactly the same way as \cite[Lemma A.3]{Dong:2023aa} resulting the estimate 
\[\begin{split}
&\|\partial_t  (\zeta_0 \bm{u})\|_{[L^p(\mathbb{R}_{T}^d)]^d} +  \|(-\Delta)^s (\zeta_0 \bm{u}) \|_{[L^p(\mathbb{R}_{T}^d)]^d} + \lambda \| \zeta_0 \bm{u}\|_{[L^p(\mathbb{R}_{T}^d)]^d} \\
&  \leq N \|\bm{f}\|_{[L^p( (0,T) \times B_R )]^d} + N \frac{\|\bm{u}\|_{[L^p((0,T) \times  B_R)]^d}}{(R-r)^{2s}} + N \frac{R^{d/p} (1+R^{d+2s})}{(R-r)^{d+2s}} \|\bm{u}\|_{[L^p((0,T); L^1(\mathbb{R}^d , \psi))]^d}.  
\end{split}
\]
We use this and Sobolev embedding \cite[Lemma A.6]{Dong:2023aa} to obtain the desired estimates. 
\end{remark}

Given $r_1,r_2>0$, we define the parabolic cylinder by 
\[
Q_{r_1,r_2}(t, \bm{x}) :=  (t - r_1^{2s}, t) \times B_{r_2}(\bm{x})  \quad \text{and} \quad Q_{r}(t, \bm{x}) = Q_{r,r}(t, \bm{x}).
\]

\begin{proposition}
\label{decomposition}
    Fix $p \in (1, \infty)$. Let $s \in (0,1), \lambda \geq 0$, and $ T \in (0, \infty)$. 
   Suppose that Theorem $\ref{timedep}$ holds for this $p$.   
    Let $\bm{u} \in \mathbb{H}^{(2s,p); 1}_0(\mathbb{R}^d_{T})$ satisfy the equation 
    \[
       \partial_t \bm{u} + (-\Delta)^s \bm{u} + \lambda \bm{u} = \bm{f} \quad \text{in} \ (0, T) \times \mathbb{R}^d .
       \]
       Then there exists $p_1 = p_1(d,s,p) \in (p, \infty]$ such that $p_1- p > \delta(d, s) >0$ and the following holds:
       for any $(t_0 , \bm{x}_0) \in \mathbb{R}^d_T$, $R>0$, and $S = \min \{0, t_0 - R^{2s}\}$ the decomposition $\bm{u} = \bm{w} + \bm{v} $ in $Q_R(t_0 , \bm{x}_0)$ holds where 
       \[
       \bm{w}  \in \mathbb{H}^{(2s,p); 1}_0( (t_0 - R^{2s}, t_0) \times \mathbb{R}^d ) ,\quad  \bm{v} \in \mathbb{H}^{(2s,p); 1}_0( (S,t_0) \times \mathbb{R}^d ), 
       \]

{\Small \begin{align} 
\label{no1}
&\left(\fint_{Q_R(t_0, \bm{x}_0)} |\mathbb{L} \bm{w}(t, \bm{x} )|^p \ d\bm{x} \ dt \right)^{1/p} + \left(\fint_{Q_R(t_0, \bm{x}_0)} |\lambda  \bm{w}(t, \bm{x} )|^p \ d\bm{x} \ dt \right)^{1/p}  \leq N \left(\fint_{Q_{2R}(t_0, \bm{x}_0)} |\bm{f}(t, \bm{x} )|^p \ d\bm{x} \ dt \right)^{1/p}
\end{align}}

and, if $p_1< \infty$,

{\small \begin{align}
\label{no2}
\left(\fint_{Q_{R/2}(t_0, \bm{x}_0)} |\mathbb{L} \bm{v}(t, \bm{x})|^{p_1} \ d\bm{x} \ dt\right)^{1/p_1}  \nonumber 
&\leq N \left(\fint_{Q_{2R}(t_0, \bm{x}_0)} |\bm{f}(t, \bm{x})|^p \ d\bm{x} \ dt \right)^{1/p}\\
& + N \sum_{k=0}^\infty 2^{-2ks} \left(\fint_{(t_0 - R^{2s}, t_0) \times B_{2^kR}(\bm{x}_0) } |\bm{f}(t, \bm{x})|^p \ d\bm{x} \ dt \right)^{1/p} \\
&+ N \sum_{k=0}^\infty 2^{-2ks} \left(\fint_{ (t_0 - R^{2s} , t_0) \times B_{2^kR}(\bm{x}_0) } |\mathbb{L} \bm{u}(t, \bm{x})|^p \ d\bm{x} \ dt \right)^{1/p}  \nonumber\\
&+ N \sum_{k=0}^\infty 2^{-2ks} \left(\fint_{ (t_0 -(2^{k+1}+1) R^{2s}, t_0) \times B_{R}(\bm{x}_0)} |\mathbb{L} \bm{u}(t, \bm{x})|^p \ d\bm{x} \ dt \right)^{1/p} ,  \nonumber 
\end{align}}
and 
{\small \begin{align}
\label{no3}
&\left(\fint_{Q_{R/2}(t_0, \bm{x}_0)} |\lambda \bm{v}(t, \bm{x})|^{p_1} \ d\bm{x} \ dt\right)^{1/p_1} \nonumber \\
&\leq N \left(\fint_{Q_{2R}(t_0, \bm{x}_0)} |\bm{f}(t, \bm{x})|^p \ d\bm{x} \ dt \right)^{1/p} + N \sum_{k=0}^\infty 2^{-2ks} \left(\fint_{ (t_0 - R^{2s}, t_0) \times B_{2^kR}(\bm{x}_0) } |\bm{f}(t, \bm{x})|^p \ d\bm{x} \ dt \right)^{1/p} \\
&+ N \sum_{k=0}^\infty 2^{-2ks} \left(\fint_{ (t_0 - R^{2s} , t_0) \times B_{2^kR}(\bm{x}_0) } |\mathbb{L} \bm{u}(t, \bm{x})|^p \ d\bm{x} \ dt \right)^{1/p} \nonumber \\
&+ N \sum_{k=0}^\infty 2^{-2ks} \left(\fint_{ (t_0 -(2^{k+1}+1) R^{2s}, t_0) \times B_{R}(\bm{x}_0) } |\mathbb{L} \bm{u}(t, \bm{x})|^p \ d\bm{x} \ dt \right)^{1/p}.  \nonumber
\end{align}}
The constant $N$ depends only on $d,s, p, \alpha_1, \alpha_2$. 

If $p_1 = \infty$, then 
$
\left(\fint_{Q_{R/2}(t_0, \bm{x}_0)} |\cdot|^{p_1} \ d\bm{x}  \ dt \right)^{1/{p_1}}
$
is replaced by 
$\|\cdot\|_{L^\infty(Q_{R/2}(t_0, \bm{x}_0))}$ 
and $\bm{u}, \bm{f}$ are extended to be zero for $t<0$. 
\end{proposition}
\begin{proof}
We assume $\bm{x}_0 =\bm{0}, R= 1$. Otherwise, for general $R>0$ and general $\bm{x}_0 \neq \bm{0}$, consider $\Tilde{\bm{u}} (t, \bm{x}) = R^{-2s} \bm{u}( R^{2s} t, R(\bm{x}-\bm{x}_0)$ and $\Tilde{\bm{f}}(t, \bm{x})= \bm{f}(R^{2s}t, R(\bm{x}-\bm{x}_0))$. Then $ \partial_t \Tilde{\bm{u}} +  (-\Delta)^s  \Tilde{\bm{u}} + \lambda \Tilde{\bm{u}} = \Tilde{\bm{f}}$ in $ (0, \Tilde{T}) \times \mathbb{R}^d $, where $\Tilde{T} = R^{-2s} T$. 
The parabolic cylinder will be $\Tilde{Q}_1(t, 0) = (t-1, t) \times B_1 $. 
Using these changes of variable,  the terms in the estimates remain unchanged. 
    
    For any $t_0 \in [0,T]$, we take a cutoff function $\zeta \in C_0^\infty( (t_0- 2^{2s}, t_0 + 2^{2s}) \times B_2 )$ satisfying 
    \[
    \zeta \in [0,1] \quad \text{and} \quad \zeta = 1 \quad \text{in} \  (t_0 -1, t_0) \times B_1 .
    \]    
    Since Theorem $\ref{timedep}$ is assumed to hold, there exists $\bm{w} \in \mathbb{H}^{(2s,p);1}_0(t_0-1,t_0) $ satisfying 
    \[
        \partial_t \bm{w} + (-\Delta)^s \bm{w} + \lambda \bm{w} =\zeta \bm{f} \quad \text{in} \  (t_0-1, t_0) \times \mathbb{R}^d
\]
    and 
      \begin{align}
      \label{west}
    \|\partial_t  \bm{w}\|_{[L^p( (t_0-1, t_0)  \times \mathbb{R}^d )]^d} +  \|\mathbb{L} \bm{w}\|_{[L^p( (t_0-1, t_0) \times \mathbb{R}^d )]^d} &+ \lambda \| \bm{w}\|_{[L^p( (t_0-1, t_0) \times \mathbb{R}^d )]^d}  \nonumber\\
      &\leq N \| \zeta \bm{f}\|_{[L^p(  (t_0-1, t_0) \times \mathbb{R}^d)]^d} \nonumber\\
      &\leq N \| \bm{f}\|_{[L^p(Q_2(t_0,0)]^d},
      \end{align}
    where $N = N(d, s, p, \alpha_1, \alpha_2)$.  Estimate $(\ref{no1})$ now follows. 
    
   Now take $S = \min(0, t_0-1) $ and $\bm{v} = \bm{u} -\bm{w}$. By taking the zero extension of $\bm{w}$ for $t < t_0 -1$, we have $\bm{w} \in \mathbb{H}^{(2s,p); 1}_0((S,t_0)\times \mathbb{R}^d)$.  Thus $\bm{v} \in \mathbb{H}^{(2s,p);1}_0((S, t_0)\times \mathbb{R}^d )$ and 
    \[
    \partial_t \bm{v} + (-\Delta)^s \bm{v} + \lambda \bm{v} = (1-\zeta) \bm{f} \quad \text{in} \ (S,t_0) \times \mathbb{R}^d  .
    \]
    Moreover, we take $ \eta \in C^\infty(\mathbb{R})$ such that 
    \[
    \eta(t) = \begin{cases} 
    1 \quad \text{when} \ t \in (t_0-(1/2)^{2s},t_0),\\
    0  \quad \text{when} \ t \in (t_0-1,t_0+1)^\complement,\\
    \end{cases}
    \]
    and 
   $
    |\eta'|\leq N(s).
    $ 
    It follows that $\bm{h}:= \eta \bm{v} \in \mathbb{H}^{(2s,p);1}_0((t_0-1,  t_0) \times \mathbb{R}^d)$ and 
    \begin{align}
    \label{eqh}
    \partial_t \bm{h} + (-\Delta)^s \bm{h} + \lambda \bm{h} = \eta' \bm{v} + \eta(1 - \zeta) \bm{f}  \quad \text{in} \ (t_0-1,t_0) \times \mathbb{R}^d.   \end{align}
       We denote the convolution by   $\bm{g}^\epsilon := \bm{g} \ast \xi^\epsilon$, where 
    \[
    \xi \in C^\infty_0(\mathbb{R}^d), \quad \text{supp}(\xi) \subset B_1, \quad \int_{\mathbb{R}^d} \xi (\bm{x}) d\bm{x}= 1, \quad \xi^{\epsilon}(\cdot) := \epsilon^{-d} \xi(\cdot/\epsilon). 
    \]
   Mollifying equation \eqref{eqh} in the $\bm{x}$-variable first, and then commuting $(-\Delta)^{s}$ and the convolution, we obtain that 
    \[
    \partial_t \bm{h}^\epsilon + (-\Delta)^s \bm{h}^\epsilon + \lambda \bm{h}^\epsilon = \eta' \bm{v}^{\epsilon} + (\eta(1 - \zeta) \bm{f})^\epsilon \quad \text{in} \ (t_0-1,t_0) \times \mathbb{R}^d.
    \]
      After noting that for each $t,$ ${\bm h}^{\epsilon}(t,\cdot) \in [C^{\infty}_{b}(\mathbb{R}^{d})]^d$, and after applying $\mathbb{L}$ on both sides of the above equation, we may use  Propositions \ref{commute} 
to commute the operators $\mathbb{L}$ and $(-\Delta)^{s}$ to obtain, 
 \[
   \partial_t ( \mathbb{L} \bm{h}^\epsilon) + (-\Delta)^s (\mathbb{L} \bm{h}^\epsilon) + \lambda \mathbb{L} \bm{h}^\epsilon = \eta' \mathbb{L}  \bm{v}^\epsilon + \mathbb{L}( \eta[(1 - \zeta) \bm{f}]^\epsilon) \quad \text{in} \ (t_0-1,t_0) \times \mathbb{R}^d.
       \]
    Next, we take $p_1$ such that 
    \[
    \frac{1}{p_1} = \frac{1}{p} - \frac{s}{d+2s} \quad \text{if} \ p \leq \frac{d}{2s} +1, \quad \text{and}\quad p_1 = \infty \quad \text{if} \ p > \frac{d}{2s} +1.
    \]
Note that if $p \leq \frac{d}{2s} +1$, then  $p_1 - p  = {s p^{2}\over d + 2s - sp}> {s \over d+s} =:\delta(d,s) >0$.  

Suppose $p_1 < \infty$.   We begin by noting that since ${\bm h}\in \mathbb{H}^{(2s,p);1}_0((t_0-1,  t_0) \times \mathbb{R}^d)$ solves  \eqref{eqh} with a right hand side that is in $L^{p}$, the estimates of Theorem $\ref{timedep}$ hold for $\bm{h}$ for this $p$. We then have  $\mathbb{L} \bm{h} \in [L^p[( (S, t_0) \times \mathbb{R}^d )]^d$. Moreover,  again since Theorem $\ref{timedep}$ holds for  $p$, we may apply Corollary \ref{Lbdd} to conclude that $\mathbb{L}$ is bounded from $[H^{2s,p}(\mathbb{R}^d)]^d \to [L^{p}(\mathbb{R}^{d})]^d$. Thus for each $t\in (t_0 - 1, t_0)$ we can apply Proposition \ref{commutestar}, to commute $\mathbb{L}$  and the convolution so that $\mathbb{L} \bm{h}^\epsilon = [ \mathbb{L} \bm{h}]^ \epsilon \in [L^{p}((t_0 - 1, t_0); C_b^{\infty}(\mathbb{R}^{d}))]^d$. This implies that  a restriction of $\mathbb{L} \bm{h}^\epsilon$ over a bounded set can have an extension that will belong to  $\mathbb{H}^{(2s,p); 1}_0( (t_0-1, t_0) \times \mathbb{R}^d )$. From this and Lemma  \ref{cor},
    we have that for some $N = N(d, s, p, \alpha_1, \alpha_2)$
    \begin{equation} \label{epsilon-level}
    \begin{split}
        \|\mathbb{L} \bm{h}^\epsilon\|_{[L^{p_1}(Q_{1/2}(t_0,0))]^d} &\leq 
        \|\mathbb{L} \bm{h}^\epsilon\|_{[L^{p_1}( (t_0-1,t_0) \times B_{1/2} )]^d}\\
        & \leq N \| \eta' \mathbb{L}  \bm{v}^\epsilon + \mathbb{L}( \eta[(1 - \zeta) \bm{f}]^\epsilon)\|_{[L^{p}((t_0-1,t_0) \times B_{3/4})]^d}\\
&+ N \frac{\| \mathbb{L} \bm{h}^\epsilon \|_{[L^{p}( (t_0-1,t_0) \times B_{3/4})]^d}}{(1/4)^{2s}}\\
&+ N \frac{(3/4)^{d/p}(1+ (3/4)^{d+2s})}{(1/4)^{d+2s}} \|\mathbb{L} \bm{h}^\epsilon\|_{[L^p((t_0-1,t_0); L^1(\mathbb{R}^d , \psi))]^d} . 
      \end{split}
    \end{equation}
We will estimate the terms in the right hand side. 
Using Young's inequality for convolutions, the fact that $[(1 -\zeta){\bm f}]^{\epsilon}$ vanishes in $(t_0-1,t_0) \times B_{1-\epsilon},$ and Minkowski's inequality, 
 for sufficiently small $\epsilon$ we have that 
    \begin{align*}
      &   \|\mathbb{L}( \eta[(1 - \zeta) \bm{f}]^\epsilon) \|_{[L^{p}((t_0-1,t_0) \times B_{3/4})]^d} \\
         &\leq  \left\|\int_{\mathbb{R}^d} \frac{\bm{y} \otimes \bm{y} }{|\bm{y}|^2} \eta(\cdot)((1 - \zeta) \bm{f})^{\epsilon}(\cdot, \cdot+\bm{y}) K(\bm{y}) \ d\bm{y} \right\|_{[L^{p}((t_0-1,t_0) \times B_{3/4})]^d} \nonumber\\ 
         &=  \left\|\int_{|\bm{y}| > {1\over 4}-\epsilon} \frac{\bm{y} \otimes \bm{y} }{|\bm{y}|^2}  \eta(\cdot) ((1 - \zeta) \bm{f})^{\epsilon}(\cdot, \cdot+\bm{y})  K(\bm{y}) \ d\bm{y} \right\|_{[L^{p}((t_0-1,t_0) \times B_{3/4})]^d} \nonumber\\ 
         &\lesssim \left\|\int_{|\bm{y}| > {1\over 4}-2\epsilon}{|\bm{f}(\cdot, \cdot+\bm{y})|\over |\bm{y}|^{d + 2s}} \ d\bm{y} \right\|_{[L^{p}( (t_0-1,t_0) \times B_{3/4})]^d} \nonumber\\ 
          &\lesssim \left\| \sum_{k =-2}^\infty 2^{-kd-2ks} \int_{B_{2^k} \setminus B_{2^{k-1}}}  |\bm{f}(\cdot, \cdot + \bm{y})| \ d\bm{y} \right\|_{[L^{p}((t_0-1,t_0) \times B_{3/4} )]^d} \nonumber\\ 
          &\leq \sum_{k =0}^\infty 2^{-2ks}  \left( \fint_{ (t_0-1,t_0) \times B_{2^k}} |\bm{f}(t, \bm{x} )|^p \ d\bm{x} \ dt  \right)^{\frac{1}{p}} \nonumber,
    \end{align*}
    i.e. the term $\|\mathbb{L}( \eta[(1 - \zeta) \bm{f}]^\epsilon) \|_{[L^{p}( (t_0-1,t_0) \times B_{3/4})]^d}$ is uniformly bounded. Thus, \eqref{epsilon-level} can be estimated as 
     \begin{align*}
         &\|\mathbb{L} \bm{h}^\epsilon\|_{[L^{p_1}(Q_{1/2}(t_0,0)]^d} \\
         &\leq   N \| \eta' \mathbb{L}  \bm{v}^\epsilon\|_{[L^{p}((t_0-1,t_0) \times B_{3/4})]^d} + N  \sum_{k =0}^\infty 2^{-2ks}  \left( \fint_{ (t_0-1, t_0) \times B_{2^k}} |\bm{f}(t, \bm{x})|^p \ d\bm{x} \ dt \right)^{\frac{1}{p}}\nonumber\\
&+ N \frac{\| \mathbb{L} \bm{h}^\epsilon \|_{[L^{p}((t_0-1,t_0) \times B_{3/4})]^d}}{(1/4)^{2s}}+ N \frac{(3/4)^{d/p}(1+ (3/4)^{d+2s})}{(1/4)^{d+2s}} \|\mathbb{L} \bm{h}^\epsilon\|_{[L^p((t_0-1,t_0); L^1(\mathbb{R}^d , \psi))]^d}.        \end{align*}
Now since  ${\bm h}^{\epsilon} \to {\bm h}$ in $ \mathbb{H}^{(2s,p);1}_0((t_0-1,  t_0) \times \mathbb{R}^d)$ and that $\mathbb{L}$ is bounded from $[H^{2s,p}(\mathbb{R}^d)]^d \to [L^{p}(\mathbb{R}^{d})]^d$, we have that $\mathbb{L}{\bm h}^{\epsilon} \to \mathbb{L}{\bm h}$ in $L^{p}((t_0-1, t_0) \times \mathbb{R}^d).$  In particular, we have almost everywhere pointwise convergence. Hence, we can take the limit as $\epsilon \to 0$  and apply Fatou's lemma to obtain that 
  {\small    \begin{align}
     \label{epsilonzero}
      \|\mathbb{L} \bm{h}\|_{[L^{p_1}(Q_{1/2}(t_0,0))]^d} &\leq   N \| \eta' \mathbb{L}  \bm{v}\|_{[L^{p}((t_0-1,t_0) \times B_{3/4})]^d} + N  \sum_{k =0}^\infty 2^{-2ks}  \left( \fint_{ (t_0-1, t_0) \times B_{2^k}} |\bm{f}(t, \bm{x})|^p \ d\bm{x} \ dt \right)^{\frac{1}{p}}\nonumber\\
&+ N \| \mathbb{L} \bm{h} \|_{[L^{p}( (t_0-1,t_0) \times B_{3/4})]^d}+ N \|\mathbb{L} \bm{h}\|_{[L^p((t_0-1,t_0); L^1(\mathbb{R}^d , \psi))]^d}
        \end{align}}
      where $N = N(d, s, p, \alpha_1, \alpha_2)$.  
Since by definition $\bm{h} = \eta \bm{v}$ we can rewrite the above as, 
\begin{align}\label{epsilonzerop}
          \|\mathbb{L} \bm{v}\|_{[L^{p_1}(Q_{1/2}(t_0,0))]^d} &\leq  N  \sum_{k =0}^\infty 2^{-2ks}  \left( \fint_{(t_0-1, t_0) \times B_{2^k}} |\bm{f}(t, \bm{x})|^p \ d\bm{x} \ dt \right)^{\frac{1}{p}} \nonumber\\
&+ N \| \mathbb{L} \bm{v} \|_{[L^{p}( (t_0-1,t_0) \times B_{1})]^d}+ N \|\mathbb{L} \bm{v}\|_{[L^p((t_0-1,t_0); L^1(\mathbb{R}^d , \psi))]^d}. 
        \end{align}
  We use H\"older's inequality to estimate the last term in \eqref{epsilonzerop} as 
 \begin{align*}
      \|\mathbb{L} \bm{v}\|_{[L^p((t_0-1,t_0); L^1(\mathbb{R}^d , \psi))]^d }&= \left( \int_{t_0-1}^{t_0} \left(\int_{\mathbb{R}^d} |\mathbb{L} \bm{v}(t, \bm{x})| \psi(\bm{x}) \ d \bm{x} \right)^p \ dt \right)^{1/p}\\
&\leq \int_{\mathbb{R}^d} \left( \int_{t_0-1}^{t_0} |\mathbb{L} \bm{v}(t, \bm{x})|^p \ dt \right)^{1/p} \psi(\bm{x}) \ d \bm{x}  = I_1 + I_2,  
     \end{align*}
     where $I_1 =  \int_{B_1} \left( \int_{t_0-1}^{t_0} |\mathbb{L} \bm{v}(t, \bm{x})|^p \ dt \right)^{1/p} \psi(\bm{x}) \, d\bm{x}$ and $I_2 =  \int_{\mathbb{R}^{d}\setminus B_1} \left( \int_{t_0-1}^{t_0} |\mathbb{L} \bm{v}(t, \bm{x})|^p \ dt \right)^{1/p} \psi(\bm{x}) \, d\bm{x}$.  
     We can estimate $I_1$ and $I_2$ as follows:  
     \[
     I_1 \leq  \left( \int_{B_{1}}  \int_{t_0-1}^{t_0} |\mathbb{L} \bm{v}(t, \bm{x})|^p \ dt\, \psi(\bm{x}) \, d\bm{x} \right)^{1/p}\left( \int_{B_1}\psi(\bm{x})  \, d\bm{x} \right)^{1/p'}, 
     \]
     whereas, 
     \[
     \begin{split}
     I_2 &=   \sum_{k=1}^\infty \int_{B_{2^k} \setminus B_{2^{k-1}}} \left( \int_{t_0-1}^{t_0} |\mathbb{L} \bm{v}(t, \bm{x})|^p \ dt \right)^{1/p} \psi(\bm{x}) \, d \bm{x}\\
    &\leq  \sum_{k=1}^\infty \left( \int_{B_{2^k} \setminus B_{2^{k-1}}} \int_{t_0-1}^{t_0} |\mathbb{L} \bm{v}(t, \bm{x})|^p \ dt \,\psi(\bm{x})  \ d\bm{x} \right)^{1/p}\left( \int_{B_{2^k} \setminus B_{2^{k-1}}}\psi(\bm{x})  \ d\bm{x} \right)^{1/p'}\\
    & \leq \sum_{k=1}^\infty 2^{-2sk/p'} \left(2^{-2ks} \fint_{ (t_0-1,t_0) \times B_{2^k}} |\mathbb{L} \bm{v}(t, \bm{x})|^p \ d\bm{x} \ dt  \right)^{\frac{1}{p}}. 
     \end{split}
     \]
     Combing the estimates for $I_1$ and $I_2 $ we obtain that
     \[
       \|\mathbb{L} \bm{v}\|_{[L^p((t_0-1,t_0); L^1(\mathbb{R}^d , \psi))]^d } \leq \sum_{k=0}^\infty 2^{-2sk/p'} \left(2^{-2ks} \fint_{ (t_0-1,t_0) \times B_{2^k}} |\mathbb{L} \bm{v}(t, \bm{x})|^p \ d\bm{x} \ dt  \right)^{\frac{1}{p}}
     \]
     Using the relation  $\bm{v} = \bm{u} - \bm{w}$ and $(\ref{west})$, we obtain 
  {\small   \begin{align}
    \label{estim12}
       &\|\mathbb{L} \bm{v}\|_{[L^p((t_0-1,t_0); L^1(\mathbb{R}^d , \psi))]^d} \nonumber\\
&\leq N \sum_{k=0}^\infty  2^{-2ks} \left(\fint_{ (t_0-1,t_0) \times B_{2^k}} |\mathbb{L} \bm{v}(t, \bm{x})|^p \ d\bm{x} \ dt  \right)^{\frac{1}{p}} \nonumber\\
&\leq N \sum_{k=0}^\infty  2^{-2ks} \left(\fint_{ (t_0-1,t_0) \times B_{2^k}} |\mathbb{L} \bm{u}(t, \bm{x})|^p \ d\bm{x} \ dt  \right)^{\frac{1}{p}}+  N \sum_{k=0}^\infty  2^{-2ks} \left(\fint_{ (t_0-1,t_0) \times B_{2^k}} |\mathbb{L} \bm{w}(t, \bm{x})|^p \ d\bm{x} \ dt  \right)^{\frac{1}{p}} \nonumber\\
&\leq N  \sum_{k=0}^\infty 2^{-2ks} \left(\fint_{ (t_0-1,t_0) \times B_{2^k} } |\mathbb{L} \bm{u}(t, \bm{x})|^p \ d\bm{x} \ dt  \right)^{\frac{1}{p}} +  N   \left(\fint_{Q_{2}(t_0,0)} |\bm{f}(t, \bm{x})|^p \ d\bm{x} \ dt  \right)^{\frac{1}{p}},
    \end{align}}
       where we have used the estimate \eqref{west} for ${\bm w}$ and $N = N(d, s, p, \alpha_1, \alpha_2)$.  Combining $(\ref{epsilonzero})$ and $(\ref{estim12})$, we obtain $(\ref{no2})$ for $R=1$. 

In the event $p_1 = \infty$, repeat the same computations using point (2) of Lemma $\ref{cor}$ to replace the left hand side of \eqref{epsilon-level} by the $L^{\infty}$-norm using the inequality  $\|\mathbb{L} \bm{h}^\epsilon\|_{[L^{\infty}(Q_{1/2} (t_0,0))]^d} \leq  \|\mathbb{L} \bm{h}^\epsilon\|_{[C^{\tau/2s,\tau}(Q_{1/2}(t_0,0))]^d}$.
\end{proof}

   Once we have local estimates that are proved in the previous theorem, what remains is the level set estimate for $\mathbb{L}\bm{u}$. The argument is exactly the same as the one used in \cite{Dong:2023aa}. We include it here for completeness.  
   
   Given $\bm{f}$ a locally integrable vector-valued function on $\mathbb{R}^{d+1}$, its uncentered Hardy-Littlewood maximal function is defined as
       \[
       \mathcal{M}(\bm{f})(t, \bm{x}) = \sup_{Q_{r}(s, \bm{y}) \ni(t, \bm{x})} \fint_{Q_r(s, \bm{y})} |\bm{f}(r, \bm{z})| \ d\bm{z} \ dr
       \]
       and its strong maximal function is defined as  
       \[
       \mathcal{SM}(\bm{f}) (t, \bm{x}) = \sup_{Q_{r_1,r_2}(s, \bm{y})\ni (t, \bm{x})} \fint_{Q_{r_1,r_2}(s, \bm{y})} |\bm{f}(r, \bm{z})| \ d\bm{z} \ dr,
       \] 
where for $r_1, r_2>0$, $Q_{r_1,r_2}(t,\bm{x}) = (t-r_{1}^{2s}, t) \times B_{r_2}(\bm{x}) $. 
      For $1<p<\infty$, the maximal function is a bounded operator from $L^p(\mathbb{R}^d)$ to $L^p(\mathbb{R}^d)$ (see  \cite[Theorem 2.1.6]{grafakos2014classical}).
      
       For $R > 0$, $p \in (1, \infty)$ and $p_1 \in (p, \infty]$ from Proposition \ref{decomposition}, denote the level  superset of $\mathbb{L} \bm{u}$ by 
       \[
       \mathcal{A}(\tau) = \{(t, \bm{x}) \in  (-\infty, T) \times \mathbb{R}^d  : |\mathbb{L} \bm{u}(t, \bm{x}) | > \tau\}
       \]
       and for $\gamma>0$ 
       \[
       \mathcal{B_\gamma}(\tau) = \{ (t, \bm{x}) \in  (-\infty, T) \times \mathbb{R}^d : \gamma^{-1/p} (\mathcal{M} |\bm{f}|^p (t, \bm{x}))^{1/p} + \gamma^{-1/p_1} (\mathcal{SM} |\mathbb{L} \bm{u}|^p (t, \bm{x}))^{1/p} >\tau\}. 
       \]
       Also, we define 
       \[
       \mathcal{C}_R(t, \bm{x}) =  (t- R^{2s}, t+ R^{2s}) \times  B_R(\bm{x}) \quad \text{and} \quad \widehat{\mathcal{C}_R} (t, \bm{x}) =  \mathcal{C}_R(t, \bm{x}) \cap \{t \leq T\}. 
       \]         

       \begin{lemma}
       \label{balls}
           Let $\gamma \in (0,1)$, $s \in (0,1)$, $T \in (0, \infty), R>0, \lambda \geq 0$ and $p \in (1, \infty)$. Suppose Theorem \ref{timedep} holds for this $p$, and $\bm{u} \in \mathbb{H}^{(2s,p);1}_0(\mathbb{R}_T^d)$ satisfies \[
       \partial_t \bm{u} + (-\Delta)^s \bm{u}  + \lambda \bm{u} = \bm{f} \quad \text{in} \  (0,T) \times \mathbb{R}^d.
       \]
       Let $p_1 \in (p, \infty]$ be as obtained in  Proposition \ref{decomposition}. 
       Then, there exists a sufficiently large constant $\kappa>1$ such that for any $(t_0, \bm{x}_0) \in (-\infty, T] \times \mathbb{R}^d $ and any $\tau>0$, if 
       \[
       |\mathcal{C}_{\Tilde{R}} (t_0, \bm{x}_0) \cap \mathcal{A}(\kappa \tau)| \geq \gamma |\mathcal{C}_{\Tilde{R}} (t_0, \bm{x}_0)|,
       \]
       then we have 
       \[
       \widehat{\mathcal{C}_{\Tilde{R}}} (t_0, \bm{x}_0) \subset \mathcal{B}_\gamma(\tau),
       \]
       where $\Tilde{R} = 2^{-1-1/2s} R$. 
       
       \end{lemma}

       \begin{proof}
           Assume $\tau =1$.  Otherwise, divide the equation $ \partial_t \bm{u} + (-\Delta)^s \bm{u}  + \lambda \bm{u} = \bm{f}$ by $\tau$.     
  
         Note that if $t_0 + \Tilde{R}^{2s} <0$, since $\bm{u}$ vanishes for $t<0$, we have that  
         \[
         \mathcal{C}_{\Tilde{R}} (t_0, \bm{x}_0) \cap \mathcal{A}(\kappa ) \subset \left\{ (t, \bm{x}) \in  (-\infty, 0) \times \mathbb{R}^d  : |\mathbb{L} \bm{u}(t, \bm{x}) |>1 \right\} = \emptyset. 
         \]
     Thus, it suffices to consider the case $t_0 + \Tilde{R}^{2s} \geq 0$. We argue by contradiction.  
        Suppose that for any $\kappa >0$ there exist $(t_0, \bm{x}_0) = (t_0(\kappa), \bm{x}_{0}(\kappa))$ such that $|\mathcal{C}_{\Tilde{R}} (t_0, \bm{x}_0) \cap \mathcal{A}(\kappa)| \geq \gamma |\mathcal{C}_{\Tilde{R}} (t_0, \bm{x}_0)|,$ but for  some $(\hat{t}, \hat{\bm{y}}) = (\hat{t}(\kappa), \hat{\bm{y}}({\kappa})) \in \widehat{\mathcal{C}_{\Tilde{R}}} (t_0, \bm{x}_0)$, $(s, \bm{y}) \not \in \mathcal{B}_{\gamma}(1)$; i.e. 
           \[
           \gamma^{-1/p} (\mathcal{M} |\bm{f}|^p (\hat{t}, \hat{\bm{y}}))^{1/p} + \gamma^{-1/p_1} (\mathcal{SM} |\mathbb{L}\bm{u}|^p (\hat{t}, \hat{\bm{y}}))^{\frac{1}{p}} \leq 1. 
           \]
           In particular,
           \begin{align}
           \label{maximal}
           (\mathcal{M} |\bm{f}|^p (\hat{t}, \hat{\bm{y}}))^{1/p} \leq  \gamma^{1/p} \quad \text{and} \quad  (\mathcal{SM} |\mathbb{L}\bm{u}|^p (\hat{t}, \hat{\bm{y}}))^{\frac{1}{p}} \leq \gamma^{1/p_1}, 
           \end{align}
           where we understand that when $p_1 = \infty$, then ${1\over p_1 } = 0.$ 
           Let $t_1 := \min \{t_0+ \Tilde{R}^{2s} , T\}$. By Proposition \eqref{decomposition}, for $S = \min\{0, t_1- R^{2s}\}$,  there exist $\bm{w} \in  \mathbb{H}^{(2s,p); 1}_0((t_1 - R^{2s}, t_1) \times \mathbb{R}^d ) $ and $\bm{v} \in \mathbb{H}^{(2s,p); 1}_0((S,t_1) \times \mathbb{R}^d )$ such that $\bm{u} = \bm{v} + \bm{w}$ in $Q_R(t_1, \bm{x}_0)$ and estimates \eqref{no1} and \eqref{no2} hold. Now since  $(\hat{t}, \hat{\bm{y}})  \in \widehat{\mathcal{C}_{\Tilde{R}}}(t_0, \bm{x}_0)  \subset Q_{R/2}(t_1, \bm{x}_0) \subset Q_{2R}(t_1, \bm{x}_0) $ and $(\hat{t}, \hat{\bm{y}})  \in \widehat{\mathcal{C}_{\Tilde{R}}}(t_0, \bm{x}_0)  \subset (t_1 -(2^{k+1} +1) R^{2s} , t_1) \times  B_R(\bm{x}_0)$  and also $(\hat{t}, \hat{\bm{y}}) \in \widehat{\mathcal{C}_{\Tilde{R}}}(t_0, \bm{x}_0)  \subset (t_1 - R^{2s} , t_1) \times B_{2^kR}(\bm{x}_0)$  for all $k=\mathbb{N}\cup\{0\}, $ from the estimates \eqref{no1} , \eqref{no2} and \eqref{maximal}, we get:
  \[
\left(\fint_{Q_{R}(t_1, \bm{x}_0)} |\mathbb{L}\bm{w}(t, \bm{x})|^p \ d\bm{x} \ dt \right)^{1/p} 
\leq N \gamma^{1/p}
\]
and
\[
\left(\fint_{Q_{R/2}(t_1, \bm{x}_0)} |\mathbb{L} \bm{v}(t, \bm{x})|^{p_1} \ d\bm{x} \ dt \right)^{1/p_1}  \leq N \gamma^{1/p_1}.
\]
If  $p_1 < \infty$,  by Chebyshev's inequality and for a constant $C_1$ to be determined we have 
           \begin{align*}
           |\mathcal{C}_{\Tilde{R}} (t_0, \bm{x}_0) \cap \mathcal{A}(\kappa)| =& \  |\{ (t, \bm{x}) \in \mathcal{C}_{\Tilde{R}} (t_0, \bm{x}_0), t \in (-\infty, T) : | \mathbb{L} \bm{u} (t, \bm{x}) | > \kappa \}|\\
           \leq& \  |\{ (t, \bm{x}) \in Q_{R/2} (t_1, \bm{x}_0) : | \mathbb{L} \bm{u} (t, \bm{x}) | > \kappa\}|\\
           \leq&  \  |\{ (t, \bm{x}) \in Q_{R/2} (t_1, \bm{x}_0) : | \mathbb{L} \bm{w} (t, \bm{x}) | > \kappa- C_1 \}| + |\{ (t, \bm{x}) \in Q_{R/2} (t_1, \bm{x}_0) : | \mathbb{L} \bm{v} (t, \bm{x}) | > C_1 \}|\\
           \leq&  \int_{Q_{R/2}(t_1, \bm{x}_0) } \ (\kappa-C_1)^{-p}  | \mathbb{L} \bm{w} (t, \bm{x}) |^{p} + C_1^{-p_1}  | \mathbb{L} \bm{v} (t, \bm{x}) |^{p_1}  d\bm{x} \ dt   \\
           \leq& \ \frac{N^p \gamma |Q_{R/2}|}{(\kappa - C_1)^p} + \frac{N^{p_1} \gamma |Q_{R/2}|}{C_1^{p_1}} \\
           \leq& \ N_0(d, s) |\mathcal{C}_{\Tilde{R}(t_0, \bm{x}_0)}| \gamma \left(\frac{N^p}{(\kappa - C_1)^p } + \frac{N^p}{C_1^p} \right).
           \end{align*}
  By first taking a sufficiently large $C_1$ such that $N_0(N /C_1)^p < 1/2$,
and then taking a large $\kappa$ so that
$N_0 N^p/ (\kappa- C_1)^p < 1/2$, we obtain 
\[
|\mathcal{C}_{\Tilde{R}} (t_0, \bm{x}_0) \cap \mathcal{A}(\kappa)|< \gamma |\mathcal{C}_{\Tilde{R}} (t_0, \bm{x}_0)|,
\]
which is a contradiction.  

If $p_1 = \infty$ instead, then from \eqref{maximal} we have $ (\mathcal{SM} |\mathbb{L} \bm{u}|^p (\hat{t}, \hat{\bm{y}}))^{\frac{1}{p}} \leq 1$ and by Chebyshev's inequality, 
           \begin{align*}
           |\mathcal{C}_{\Tilde{R}} (t_0, \bm{x}_0) \cap \mathcal{A}(\kappa)| 
           \leq& \ \kappa^{-p}  \int_{Q_{R/2}(t_1, \bm{x}_0)}   | \mathbb{L} \bm{u} (t, \bm{x}) |^{p} d\bm{x}  \ dt \leq  \ \frac{N^p  |\mathcal{C}_{\Tilde{R}(t_0, \bm{x}_0)}| }{\kappa ^p}. 
           \end{align*}
           By taking a large $\kappa$ so that
$ N^p/ \kappa^p < \gamma$, we obtain 
\[
|\mathcal{C}_{\Tilde{R}} (t_0, \bm{x}_0) \cap \mathcal{A}(\kappa)|< \gamma |\mathcal{C}_{\Tilde{R}} (t_0, \bm{x}_0)|,
\]
which is a contradiction again.  
       \end{proof}

       \begin{remark}
       \label{remarkballs}
           By replacing $\mathbb{L} \bm{u}$ with $\lambda \bm{u}$ in the above proof and using \eqref{no1} and \eqref{no3} from Proposition \ref{decomposition} we conclude that Lemma $\ref{balls}$ holds for 
           \[
           \mathcal{A}'(s) = \{(t, \bm{x}) \in  (-\infty, T) \times  \mathbb{R}^d  : |\lambda \bm{u} (t, \bm{x})| > \tau\}
           \]
           and 
            \[
       \mathcal{B}'_\gamma(s) = \{ (t, \bm{x}) \in  (-\infty, T) \times \mathbb{R}^d  : \gamma^{-1/p} (\mathcal{M} |\bm{f}|^p (t, \bm{x}))^{1/p} + \gamma^{-1/p_1} (\mathcal{SM} |\lambda \bm{u}|^p (t, \bm{x}))^{\frac{1}{p}} >\tau\}.  
       \]
       \end{remark}

       \begin{proof}[Proof of Theorem \ref{timedep}]
We first assume that  $\bm{u} \in [C_{0}^\infty(\mathbb{R}^d_{T})]^d$ with $\bm{u}(0, \cdot) =0$ solves \eqref{eqtimedep}. 

We first establish the result for $p \in [2,\infty)$ via a bootstrap argument by starting from $p=2$ and use an iterative argument to successively increase the exponent to any $p$.  

Now we assume that the theorem holds for some
$p_0 \in [2,\infty)$, and we prove the estimate $(\ref{est1})$ for $ p\in (p_0, p_1)$, where $p_1$ comes from Proposition \ref{decomposition}. Note that after change of variable 
\[
\|\mathbb{L} \bm{u}\|^p_{[L^p(\mathbb{R}_T^d)]^d} = p \int_0^\infty |\mathcal{A}(\tau) | \tau^{p-1} \ d\tau = p \kappa^p \int_0^\infty |\mathcal{A}(\kappa \tau) \tau^{p-1} \ d\tau. 
\]
Moreover, from Lemma \ref{balls} and the Lemma of Crawling of ink spot (Lemma A.5 in \cite{Dong:2023aa}), we get
\[
|\mathcal{A}(\kappa \tau)| \leq N(d) \gamma |\mathcal{B}_\gamma (\tau)|
\]
for all $\tau \in (0, \infty)$, where $\kappa$ is from Lemma \ref{balls}. Thus, by the Hardy-Littlewood theorem for (strong) maximal functions,
\begin{align*}
\|\mathbb{L} \bm{u}\|^p_{[L^p(\mathbb{R}_T^d)]^d} &\leq N p \kappa^p \gamma \int_0^\infty |\mathcal{B}_\gamma(\tau)| \tau^{p-1} \ d\tau\\
&\leq  N \gamma \int_0^\infty \left| \left\{ (t, \bm{x}) \in  (-\infty, T) \times \mathbb{R}^d : \gamma^{-1/p_1} (\mathcal{SM} |\mathbb{L} \bm{u}|^{p_0} (t, \bm{x}))^{1/p_0} > \tau/2 \right\} \right| \ \tau^{p-1} \ d\tau \\
&+  N \gamma \int_0^\infty \left| \left\{ (t, \bm{x}) \in (-\infty, T) \times \mathbb{R}^d: \gamma^{-1/p_0} (\mathcal{M} |\bm{f}|^{p_0} (t, \bm{x}))^{1/p_0} > \tau/2 \right\} \right| \ \tau^{p-1} \ d\tau \\
&\leq  N \gamma^{1-p/p_1} \|\mathbb{L} \bm{u} \|^p_{[L^p(\mathbb{R}^d_{T})]^d} + N \gamma^{1-p/p_0} \| \bm{f} \|^p_{[L^p(\mathbb{R}^d_{T})]^d}.  
\end{align*}
Similarly, using Remark $\ref{remarkballs}$, we can conclude that 
\[
\|\lambda \bm{u}\|^p_{[L^p(\mathbb{R}^d_{T})]^d} \leq N \gamma^{1-p/p_1} \|\lambda \bm{u}\|^p_{[L^p(\mathbb{R}^d_{T})]^d} + N \gamma^{1- p/p_0} \| \bm{f}\|^p_{[L^p(\mathbb{R}^d_{T})]^d}.  
\]
Note that since $\bm{u}\in [C_0^{\infty}(\mathbb{R}^{d}_{T})]^{d}$, then both $\mathbb{L} \bm{u}$ and $\lambda \bm{u}$ are in  $[L^p(\mathbb{R}^d_{T})]^d$.  By the fact that $p < p_1$, we take a sufficiently small $\gamma \in (0,1)$ so that
\[
N \gamma^{1-p/p_1}< 1/2,
\]
and therefore 
\[
\|\mathbb{L} \bm{u}\|^p_{[L^p(\mathbb{R}_T^d)]^d} +\|\lambda \bm{u}\|^p_{[L^p(\mathbb{R}^d_{T})]^d}\leq  N \| \bm{f}\|^p_{[L^p(\mathbb{R}^d_{T})]^d}. 
\]
But since $\bm{u}$ solves \eqref{eqtimedep}, we have $\partial_{t}\bm{u} = \bm{f} - (-\Delta)^s \bm{u}-\lambda \bm{u}$,  and thus $\| \partial_{t}\bm{u}\|_{[L^p(\mathbb{R}_T^d)]^d}$ can now be bounded be bounded by  $\| \bm{f}\|^p_{[L^p(\mathbb{R}^d_{T})]^d}$, 
which then implies $(\ref{esttimedep})$ for $p \in (p_0, p_1)$.   
We repeat this procedure. Recall that $p_1- p$ depends only on $d,s$. Thus in finite steps, we get a $p_0$ which is larger than $d/2+1$, so that
$p_1 =\infty$. Therefore, the theorem is proved for any $p \in [2,\infty)$ for smooth functions. 

We next prove \eqref{esttimedep} for $\bm{u}\in \mathbb{H}^{(2s, p); 1}_{0}(\mathbb{R}^{d}_{T})$ that solves \eqref{eqtimedep}. 
Once we have established \eqref{esttimedep}  for smooth vector functions, then we may apply Remark \ref{Lbdd-time}, a consequence of Corollary \ref{Lbdd}, to conclude that 
\[
\left\|\mathbb{L} \bm{u} \right\|_{[L^p (\mathbb{R}_{T}^d )]^d}  \leq N \left\|(-\Delta)^s \bm{u}\right\|_{[L^p (\mathbb{R}_{T}^d)]^d}, \quad \text{for any ${\bm  u}\in [C_{0}^{\infty}(\mathbb{R}^{d}_{T})]^d$ with $\bm{u}(0, \cdot) = 0$}. 
\]
Since $[C_{0}^{\infty}(\mathbb{R}^{d}_{T})]^d$ is dense in $\mathbb{H}^{(2s, p); 1}_{0}(\mathbb{R}^{d}_{T})$,  the above estimate implies $\mathbb{L}$ has a unique continuous extension from $\mathbb{H}^{(2s, p); 1}_{0}(\mathbb{R}^{d}_{T})$ to $[L^p (\mathbb{R}_{T}^d)]^d$.  Now take ${\bm  u}_{n} \in [C_{0}^{\infty}(\mathbb{R}^{d}_{T})]^d$ with $\bm{u}(0, \cdot) = 0$ such that 
\[
\|\bm{u} - \bm{u}_{n} \|_{\mathbb{H}^{(2s, p); 1}_{0}(\mathbb{R}^{d}_{T})} \to 0, \quad \text{as $n\to \infty.$ }
\]
Then ${\bm u_{n}}$ solves \eqref{eqtimedep} with the right hand side ${\bm g}_n$ that converges to ${\bm g}$ in $[L^{p}(\mathbb{R}^{d}_{T})]^{p}$. The estimate for $\bm{u}$ now follows from the estimate for the smooth case and the continuous extension of $\mathbb{L} $. 

For $p \in (1,2)$, we use a duality argument. Again, we assume that $\bm{u} \in [C^\infty_0(\mathbb{R}^d_{T})]^d $ with $ \bm{u}(0,\cdot) =0$ solves \eqref{eqtimedep} and prove \eqref{esttimedep}. For $p' = p/(p-1)$ and $\bm{\phi} \in[ L^{p'}(\mathbb{R}^d_{T})]^d$, there exists $\bm{U} \in \mathbb{H}^{2s, p'}_0( (-T,0) \times \mathbb{R}^d)$ satisfying 
\[
\partial_t \bm{U}(t, \bm{x}) + (-\Delta)^s \bm{U}(t, \bm{x}) + \lambda \bm{U}(t, \bm{x}) = \bm{\phi}(-t, \bm{x}) \quad \text{in} \  (-T,0) \times \mathbb{R}^d
\]
and
\begin{align*}
\|\partial_t  \bm{U}\|_{[L^{p'}((- T,0) \times \mathbb{R}^d ) ]^d} +  \|\mathbb{L} \bm{U}\|_{[L^{p'}((- T,0) \times \mathbb{R}^d ) ]^d} &+ \lambda \| \bm{U}\|_{[L^{p'}((- T,0) \times \mathbb{R}^d ) ]^d} \\  &\leq N \|\bm{\phi}(-t, \bm{x})\|_{[L^{p'}((- T,0) \times \mathbb{R}^d ) ]^d}\\
&\leq N \|\bm{\phi}\|_{[L^{p'}((0,T) \times \mathbb{R}^d)]^d} .
\end{align*}
It follows that 
\begin{align*}
&\int_0^T \int_{\mathbb{R}^d} \bm{\phi}(t, \bm{x}) \cdot  \mathbb{L} \bm{u}(t, \bm{x}) \ d\bm{x} \ dt\\
&= \int_{-T}^0 \int_{\mathbb{R}^d} \bm{\phi}(-t, \bm{x}) \cdot \mathbb{L} \bm{u}(-t, \bm{x}) \ d\bm{x} \ dt\\
&= \int_{-T}^0 \int_{\mathbb{R}^d}[ \partial_t \bm{U}(t, \bm{x}) + (-\Delta)^s \bm{U}(t, \bm{x}) + \lambda \bm{U}(t, \bm{x}) ] \cdot \mathbb{L} \bm{u}(-t, \bm{x}) \ d\bm{x} \ dt.
\end{align*}
If $\bm{U}$ is smooth, we can apply Plancherel's theorem:
\begin{align*}
 &\int_{-T}^0 \int_{\mathbb{R}^d}[ \partial_t \bm{U}(t, \bm{x}) + (-\Delta)^s \bm{U}(t, \bm{x}) + \lambda \bm{U}(t, \bm{x}) ] \cdot \mathbb{L} \bm{u}(-t, \bm{x}) \ d\bm{x} \ dt\\
 &=  \int_{-T}^0 \int_{\mathbb{R}^d}[ \partial_t \widehat{\bm{U}}(t, \bm{\xi}) - |\bm{\xi}|^{2s}  \widehat{\bm{U}}(t, \bm{\xi}) + \lambda \widehat{\bm{U}}(t, \bm{\xi}) ] \cdot \widehat{\mathbb{L} \bm{u}}(-t, \bm{\xi}) \ d\bm{\xi} \ dt. 
\end{align*}
Moreover, using integration by parts on the term $\partial_t \widehat{\bm{U}}(t, \bm{\xi}) \cdot \widehat{\mathbb{L} \bm{u}}(-t, \bm{\xi})$ and the zero initial conditions of $\bm{u}$ and $\bm{U}$, we have 
\begin{align*}
 & \int_{-T}^0 \int_{\mathbb{R}^d}[ \partial_t \widehat{\bm{U}}(t, \bm{\xi}) -|\bm{\xi}|^{2s}  \widehat{\bm{U}}(t, \bm{\xi}) + \lambda \widehat{\bm{U}}(t, \bm{\xi}) ] \cdot \widehat{\mathbb{L} \bm{u}}(-t, \bm{\xi}) \ d\bm{\xi} \ dt \\
 &=  \int_{-T}^0 \int_{\mathbb{R}^d}  \widehat{\bm{U}}(t, \bm{\xi})  \cdot \left[\partial_t \widehat{\mathbb{L} \bm{u}}(-t, \bm{\xi})    +(\lambda - |\bm{\xi}|^{2s})  \widehat{\mathbb{L} \bm{u}}(-t, \bm{\xi}) \right] \ d\bm{\xi} \ dt. 
\end{align*}
By applying Plancherel's again, we get 
\begin{align*}
 &  \int_{-T}^0 \int_{\mathbb{R}^d}  \widehat{\bm{U}}(t, \bm{\xi})  \cdot \left[\partial_t \widehat{\mathbb{L} \bm{u}}(-t, \bm{\xi})    +(\lambda - |\bm{\xi}|^{2s}) \widehat{\mathbb{L} \bm{u}}(-t, \bm{\xi}) \right] \ d\bm{\xi} \ dt\\
 &= \int_{-T}^0 \int_{\mathbb{R}^d}  \bm{U}(t, \bm{x}) \cdot \left[ \partial_t \mathbb{L} \bm{u}(-t, \bm{x})     +(\lambda + (-\Delta)^s  ) \mathbb{L} \bm{u}(-t, \bm{x})  \right] \ d\bm{x} \ dt\\
 &= \int_{-T}^0 \int_{\mathbb{R}^d}   \left[\partial_t \bm{u}(-t, \bm{x})     +(\lambda + (-\Delta)^s  )\bm{u}(-t, \bm{x})  \right] \cdot  \mathbb{L}^*  \bm{U}(t, \bm{x})  \ d\bm{x} \ dt\\
 &= \int_0^T \int_{\mathbb{R}^d}   \left[\partial_t \bm{u}(t, \bm{x})     +(\lambda + (-\Delta)^s  )\bm{u}(t, \bm{x})  \right]  \cdot  \mathbb{L}^*  \bm{U}(-t, \bm{x})  \ d\bm{x} \ dt\\
 &= \int_0^T \int_{\mathbb{R}^d}  \bm{f}(t, \bm{x}) \cdot  \mathbb{L}^*  \bm{U}(-t, \bm{x})  \ d\bm{x} \ dt\\
 &\leq N \|\bm{f}\|_{[L^{p}(\mathbb{R}_T^d )]^d}\|\bm{\phi}\|_{[L^{p'}(\mathbb{R}_T^d )]^d}. 
\end{align*}
We thus conclude that 
\begin{equation} \label{dualest}
    \int_0^T \int_{\mathbb{R}^d} \bm{\phi}(t, \bm{x}) \cdot  \mathbb{L} \bm{u}(t, \bm{x}) \ d\bm{x} \ dt\\\leq N \|\bm{f}\|_{[L^{p}(\mathbb{R}_T^d )]^d}\|\bm{\phi}\|_{[L^{p'}(\mathbb{R}_T^d )]^d}
\end{equation}
In the case in which $\bm{U}$ is not smooth, we consider a sequence $\{\bm{U}_k \} \subset [C^\infty_0( [-T,0] \times \mathbb{R}^d)]^d$ with $\bm{U}_k (-T, \cdot)=0$ such that 
\[
\bm{U}_k \to \bm{U} \quad \text{in} \ \mathbb{H}^{2s,p'}_0 ((-T,0) \times \mathbb{R}^d). 
\]
Let $\bm{\phi}_{k}(-t, \bm{x}) = \partial_t \bm{U}_k + (-\Delta)^s \bm{U}_k + \lambda \bm{U}_k. $ The $\bm{\phi}_{k} \to \bm{\phi}$ in $L^{p'}((0, T)\times \mathbb{R}^d)$.   Applying the previous estimates to each $k$, we have 
\[
\begin{split}
&\int_0^T \int_{\mathbb{R}^d} \bm{\phi}(t, \bm{x}) \cdot  \mathbb{L} \bm{u}(t, \bm{x}) \ d\bm{x} \ dt\\
&= \int_0^T \int_{\mathbb{R}^d} \bm{\phi}_k(t, \bm{x}) \cdot  \mathbb{L} \bm{u}(t, \bm{x}) \ d\bm{x} \ dt + \int_0^T \int_{\mathbb{R}^d} \left(\bm{\phi}(t, \bm{x})-\bm{\phi}_k(t, \bm{x})\right) \cdot  \mathbb{L} \bm{u}(t, \bm{x}) \ d\bm{x} \ dt\\
&\leq N \|\bm{f}\|_{[L^{p}(\mathbb{R}_T^d )]^d}\|\bm{\phi}_k\|_{[L^{p'}(\mathbb{R}_T^d )]^d} + \|\phi_k - \phi\|_{L^{p'}(\mathbb{R}^{d}_{T})}\|\lambda \bm{u} + \mathbb{L} \bm{u}\|_{L^{p}(\mathbb{R}^{d}_{T})}
\end{split}
\]
Letting $k\to \infty,$ \eqref{dualest} holds now in general.  
By replacing $\mathbb{L} \bm{u}$ with $\lambda \bm{u}$ in the previous computation, we analogously obtain
\begin{equation} \label{dualestlambda}
    \int_0^T \int_{\mathbb{R}^d} \bm{\phi}(t, \bm{x}) \cdot \lambda \bm{u} (t, \bm{x}) \ d\bm{x} \ dt\\\leq N \|\bm{f}\|_{[L^{p}(\mathbb{R}_T^d )]^d}\|\bm{\phi}\|_{[L^{p'}(\mathbb{R}_T^d )]^d}
\end{equation}
Now dividing both sides of \eqref{dualest} and \eqref{dualestlambda} by $\|\bm{\phi}\|_{[L^{p'}(\mathbb{R}_T^d )]^d}$ and taking the supremum over all nonzero $\bm{\phi}\in [L^{p'}(\mathbb{R}_T^d )]^d$, we obtain by duality that 
\begin{equation} \label{Ludual}
\|\mathbb{L} \bm{u}\|_{[L^{p}(\mathbb{R}_T^d )]^d} +\lambda \|\bm{u}\|_{[L^{p}(\mathbb{R}_T^d )]^d} \leq N \|\bm{f}\|_{[L^{p}(\mathbb{R}_T^d )]^d}.
\end{equation}
Lastly, since $\bm{u}$ solves \eqref{eqtimedep}, $\partial_t \bm{u} = \bm{f} - \lambda \bm{u} - (-\Delta)^s  \bm{u}$, which simply gives
\begin{equation} \label{partialudual}
 \|\partial_t \bm{u}\|_{[L^{p}(\mathbb{R}_T^d )]^d} \leq N \|\bm{f}\|_{[L^{p}(\mathbb{R}_T^d )]^d}
\end{equation}
By combining \eqref{Ludual} and \eqref{partialudual}, we obtain \eqref{esttimedep}.
\end{proof}
\section*{Acknowledgment}
The authors thank Prof. Hongjie Dong for bringing the manuscript \cite{Dong:2023aa} to their attention, which served as the foundation for this work.
The authors gratefully acknowledge funding from the US National Science Foundation via grants NSF-DMS 2206252 and 2509059. 
\bibliographystyle{plain}

 \bibliography{pj2.bib}

\appendix 
\section{Proof of Proposition \ref{property1}}

\begin{proof}
Fix an arbitrary bounded open set $\Omega \subset \mathbb{R}^d$. We prove existence of the integral and continuity of the map $\bm{x} \mapsto \mathbb{L} \bm{u} (\bm{x})$ on $\Omega$.

By definition $\bm{u} \in[L^1(\mathbb{R}^d,\psi)]^d$ if and only if $\psi \bm{u} \in [L^1(\mathbb{R}^d)]^d$. 
Hence, as $N \to \infty$, we have $\|\psi \bm{u} - \psi \bm{u} \chi_{[-N,N]} \|_{[L^1(\mathbb{R}^d)]^d} \to 0$. This occurs if and only if $\|\bm{u} - \bm{u} \chi_{B_N}\|_{[L^1(\mathbb{R}^d, \psi)]^d} \to 0$ as $N \to \infty$. 

Now we convolve $\bm{u}_N := \bm{u} \chi_{B_N}$ with the standard mollifier $\eta_\delta$, where $\eta_\delta(\bm{x}) = \delta^{-d} \eta(\bm{x}/\delta)$ for a smooth, nonnegative,  $\eta\in C_c^{\infty}(B_{1}(\bm{ 0}))$ satisfying 
$\int_{\mathbb{R}^d} \eta = 1$. 
Set
\[
\bm{u}_{N,\delta} := \eta_\delta * \bm{u}_N.
\]
Then $\bm{u}_{N,\delta} \in [C_c^\infty(\mathbb{R}^d)]^d$ (supp $\bm{u}_{N,\delta} \subset B_{N+\delta}$). Moreover, 
\[
\|\bm{u}_{N,\delta} - \bm{u}\|_{[L^1(\mathbb{R}^d, \psi)]^d} \leq \|\bm{u}_{N,\delta} - \bm{u}_N\|_{[L^1(\mathbb{R}^d, \psi )]^d} + \| \bm{u}_N - \bm{u}\|_{[L^1(\mathbb{R}^d, \psi )]^d} 
\]
The second term on the right hand side goes to 0 as $N \to \infty$ by construction. Let us consider the first term. For $\delta\leq 1$, we have 
\begin{align*}
\|\bm{u}_{N,\delta} - \bm{u}_N\|_{[L^1(\mathbb{R}^d, \psi )]^d}  &= \int_{\mathbb{R}^d} \psi(\bm{x}) |\bm{u}_{N,\delta}(\bm{x}) - \bm{u}_N(\bm{x})| \ d\bm{x}\\
&= \int_{B_{N+1}} \psi(\bm{x}) |\bm{u}_{N,\delta}(\bm{x}) - \bm{u}_N(\bm{x})| \ d\bm{x}\\
&\leq \|\bm{u}_{N,\delta} - \bm{u}_N\|_{[L^1(B_{N+1})]^d} \to 0 \quad \text{as } \delta \to 0 \ \text{for all } N \in \mathbb{N},
\end{align*}
where the convergence of the last term follows from the standard approximation property of the mollification, given $\bm{u}_N \in [L^1(B_{N+1})]^d$ for all $N \in \mathbb{N}$. 
We conclude that   
$\|\bm{u}_{N, \delta} - \bm{u}\|_{[L^1(\mathbb{R}^d, \psi)]^d} \to 0$ as $N \to \infty$ 
and $\delta \to 0$.  
Choose $\tilde{N}$, positive integer large enough that $\Omega\subset B_{\tilde{N}}$. Then for any $N \geq 2\tilde{N}$, the set $\overline{\Omega} + B_{1}(0) $ is in the region where  $\bm{u}_{N} = \bm{u}$ and that $\bm{u}_{N, \delta} = \bm{u}\ast\eta_{\delta}$
\[
\bm{u}_{N,\delta} \to \bm{u} \  \text{uniformly in } \overline\Omega \ \text{as } \delta \to 0 \ \text{and} \ N \to \infty.
\]
Moreover, $\{\bm{u}_{N,\delta}\}_{N, \delta}$ is uniformly bounded in $[C^{0, 2s+\epsilon}(\Omega)]^d $ if $s\leq 1/2$ or  $[C^{1, 2s+\epsilon-1}(\Omega)]^d$ if $s\in (1/2, 1)$. Indeed,  
\begin{align*}
\|\bm{u}_{N,\delta}\|_{[C^{0, 2s+\epsilon}(\Omega)]^d} &= \|\bm{u}_{N,\delta}\|_{[C^{0}(\Omega)]^d} + \sup\left\{\frac{|\bm{u}_{N,\delta}(\bm{x}) - \bm{u}_{N,\delta}(\bm{y})|}{|\bm{x} - \bm{y}|^{2s+\epsilon}}: \bm{x}\not = \bm{y}, \bm{x}, \bm{y}\in \Omega \right\}\\
&\leq \|\bm{u}_N\|_{[L^\infty(\mathbb{R}^d)]^d} \|\eta_\delta\|_{[L^1(\mathbb{R}^d)]^d} + \sup_{\bm{x}\not = \bm{y}} \frac{  \int_{\mathbb{R}^d }|\bm{u}_{N}(\bm{x}- \bm{z} ) - \bm{u}_{N}(\bm{y}-\bm{z})| \eta_\delta(\bm{z}) \ d\bm{z}  }{|\bm{x} - \bm{y}|^{2s+\epsilon}}\\
&\leq \|\bm{u}\|_{[L^\infty(\mathbb{R}^d)]^d}  +  [\bm{u}]_{[C^{0,2s+\epsilon}(\mathbb{R}^d)]^d} = \|\bm{u}\|_{[C^{0, 2s+\epsilon}(\mathbb{R}^d)]^d}
\end{align*}
Similarly, $\|\bm{u}_{N,\delta}\|_{[C^{1, 2s+\epsilon-1}(\mathbb{R}^d)]^d} \leq \|\bm{u}\|_{[C^{1, 2s+\epsilon-1}(\mathbb{R}^d)]^d}$.

Now, for each fixed $\delta, N$ and $\bm{x} \in \Omega$ the integral
\[
I_{N,\delta}(\bm{x}) := \int_{\mathbb{R}^d} \left( \frac{\bm{y} \otimes \bm{y}}{|\bm{y}|^2} \right) \left( \bm{u}_{N,\delta}(\bm{x} + \bm{y}) - \bm{u}_{N,\delta}(\bm{x}) - D[\bm{u}_{N,\delta}](\bm{x}) \bm{y} \chi^{(s)}(\bm{y}) \right) K(\bm{y}) \ d\bm{y}
\]
is absolutely convergent and agrees with $\mathbb{L} \bm{u}_{N,\delta} (\bm{x})$ because $\bm{u}_{N,\delta}$ is smooth and compactly supported. We next prove that $I_{N,\delta}(\bm{x})$ converges to the same integral for $\bm{u}$, say $I(\bm{x})$, for all $\bm{x} \in \Omega$ and that this integral is finite.

To that end, we bound the integrand  $\left( \frac{\bm{y} \otimes \bm{y}}{|\bm{y}|^2} \right) \left( \bm{u}_{N,\delta}(\bm{x} + \bm{y}) - \bm{u}_{N,\delta}(\bm{x})  - D[\bm{u}_{N,\delta}](\bm{x}) \bm{y} \chi^{(s)}(\bm{y}) \right) K(\bm{y})$ by an integrable function $\bm{h}(\bm{y})$ uniformly in $ \bm{x} \in \Omega$, for all $N \in \mathbb{N}$ and $\delta>0$. 
We treat the two regions $\{\bm{y} \leq 1\}$ and $\{\bm{y}>1\}$ separately. 
 
 {\bf The case $|\bm{y}|\leq 1$:}   We further split the proof depending on the $s$ value. 

If $s\in(0,1/2)$, then by the H\"older regularity,
\[
|\bm{u}_{N,\delta}(\bm{x}+\bm{y})-\bm{u}_{N,\delta}(\bm{x})|\leq C|\bm{y}|^{2s+\epsilon}.
\]
Hence
\[
\left|\left(\frac{\bm{y}\otimes\bm{y}}{|\bm{y}|^2}\right)\left(\bm{u}_{N,\delta}(\bm{x}+\bm{y})-\bm{u}_{N,\delta}(\bm{x})\right)K(\bm{y})\right|
\leq C \alpha_2 |\bm{y}|^{2s+\epsilon}|\bm{y}|^{-d-2s}
= C|\bm{y}|^{\epsilon-d}.
\]
Since $\epsilon>0$, the function $|\bm{y}|^{\epsilon-d}$ is integrable near $0$. 
If $s\in[1/2,1)$: then since $\chi^{(s)}(\bm{y}) = 1$ for $|\bm{y}|\leq 1$, by Taylor’s expansion, we have
\[
\bm{u}_{N,\delta}(\bm{x}+\bm{y}) - \bm{u}_{N,\delta}(\bm{x}) - D[\bm{u}_{N,\delta}](\bm{x})\bm{y} \chi^{(s)}(\bm{y})
=[\nabla \bm{u}_{N,\delta}(\bm{x}) - D[\bm{u}_{N,\delta}](\bm{x}) \chi^{(s)}(\bm{y})] \cdot \bm{y}  + \mathcal{O}(|\bm{y}|^{2s+\epsilon}),
\]
uniformly for $\bm{x} \in \Omega$. 
Recall now that $D[\bm{u}](\bm{x})$ denotes the symmetric part of $\nabla\bm{u}(\bm{x})$,
\[
D[\bm{u}](\bm{x}) =\frac{1}{2} \left(\nabla\bm{u}(\bm{x})+\nabla\bm{u}(\bm{x})^T\right),
\]
so that
\[
\nabla \bm{u}(\bm{x}) -D \bm{u}(\bm{x}) =\frac{1}{2} \left(\nabla \bm{u}(\bm{x})- \nabla\bm{u}(\bm{x})^T\right)
\]
is skew-symmetric.  
 Therefore 
\[
\left(\frac{\bm{y} \otimes\bm{y}}{|\bm{y}|^2}\right)\left(\nabla\bm{u}_{N,\delta}(\bm{x})-D[\bm{u}_{N,\delta}](\bm{x})\right)\bm{y}
\]
vanishes identically for every $\bm{y}$. Hence
\[
\left|\left(\frac{\bm{y}\otimes\bm{y}}{|\bm{y}|^2}\right)
\left(\bm{u}_{N,\delta}(\bm{x}+\bm{y})-\bm{u}_{N,\delta}(\bm{x})-D[\bm{u}_{N,\delta}](\bm{x})\bm{y} \right)K(\bm{y})\right|
\leq C|\bm{y}|^{2s+\epsilon} |\bm{y}|^{-d-2s}
= C |\bm{y}|^{\epsilon-d},
\]
which is integrable near the origin because $\epsilon>0$. 

Thus the integral over $\{|\bm{y}|\leq1\}$ is absolutely convergent and uniformly bounded for $\bm{x} \in \Omega$.

{\bf The case $|\bm{y}|>1$:}  
For $|\bm{y}|>1$, we have the following estimate:   
\begin{align*}
&\left|\left( \frac{\bm{y} \otimes \bm{y}}{|\bm{y}|^2} \right) \left( \bm{u}_{N,\delta}(\bm{x} + \bm{y}) - \bm{u}_{N,\delta}(\bm{x})  - D[\bm{u}_{N,\delta}](\bm{x}) \bm{y} \chi^{(s)}(\bm{y}) \right) K(\bm{y})\right|\\
&\leq \alpha_2 \left[|\bm{u}_{N,\delta}(\bm{x} + \bm{y})| + |\bm{u}_{N,\delta}(\bm{x})|  +| D[\bm{u}_{N,\delta}](\bm{x}) \bm{y} \chi^{(s)}(\bm{y})|\right] |\bm{y}|^{-d-2s}\\
&\leq \alpha_2 C_d  \frac{|\bm{u}_{N,\delta} (\bm{x}+ \bm{y})|}{ 1+ |\bm{x}+ \bm{y}|^{d+2s}} + \alpha_2 \frac{|\bm{u}_{N,\delta} (\bm{x})|}{|\bm{y}|^{d+2s}} + \alpha_2 \frac{| D[\bm{u}_{N,\delta}](\bm{x})|}{|\bm{y}|^{d+2s-1}} \chi^{(s)}(\bm{y}). 
\end{align*}
Since $\bm{u}_{N,\delta} \to \bm{u}$ in $[L^1(\mathbb{R}^d, \psi)]^d$, we get that the first term on the right hand side (up to a subsequence) is bounded by some $\bm{g} \in [L^1(\mathbb{R}^d, \psi)]^d$. Indeed, $\bm{u}_{N,\delta} \to \bm{u}$ in $[L^1(\mathbb{R}^d, \psi)]^d$ if and only if $\psi \bm{u}_{N,\delta} \to \psi \bm{u}$ in $[L^1(\mathbb{R}^d)]^d$. By \cite[Theorem 4.9]{Brezis2010FunctionalAS}, there exists a function $\bm{f} \in [L^1(\mathbb{R}^d)]^d$ and a subsequence $\{\psi \bm{u}_{N_k,\delta_k}\}_{k}$ so that $|\psi \bm{u}_{N_k,\delta_k}| \leq \bm{f}$ a.e, $\psi \bm{u}_{N_k,\delta_k} \to \psi \bm{u}$ as $(N_k, \delta_k) \to (\infty, 0)$. Set $\bm{g}(\bm{x}) = \frac{\bm{f}(\bm{x})}{\psi(\bm{x})}$. Then $\bm{g} \in [L^1(\mathbb{R}^d, \psi)]^d$ and $| \bm{u}_{N_k,\delta_k}| \leq \bm{g}$ a.e. It follows that 
\[
\frac{|\bm{u}_{N_k,\delta_k} (\bm{x}+ \bm{y})|}{ 1+ |\bm{x}+ \bm{y}|^{d+2s}} \leq \frac{\bm{g} (\bm{x}+ \bm{y})}{ 1+ |\bm{x}+ \bm{y}|^{d+2s}} \in [L^1(\mathbb{R}^d)]^d
\]
Using the uniform bounds for the $\{\bm{u}_{N,\delta}\}_{N,\delta}'s$, we can bound the remaining terms uniformly by $C \frac{1}{|\bm{y}|^{d+2s}}$ and $C \frac{1}{|\bm{y}|^{d+2s-1}} \chi^{(s)}(\bm{y})$, which are integrable away from the origin and uniformly bounded in $\bm{x} \in \Omega$.

We now apply the Dominated Convergence Theorem. This yields
\[
 \lim_{N\to\infty , \delta \to 0} I_{N,\delta}(\bm{x}) = \int_{\mathbb{R}^d} \left( \frac{\bm{y} \otimes \bm{y}}{|\bm{y}|^2} \right) \left( \bm{u}(\bm{x} + \bm{y}) - \bm{u}(\bm{x}) - D[\bm{u}](\bm{x}) \bm{y} \chi^{(s)}(\bm{y}) \right) K(\bm{y}) \ d\bm{y}.
\]
That is, 
\[
\lim_{N\to\infty , \delta \to 0}  -\mathbb{L}\bm{u}_{N, \delta} = \int_{\mathbb{R}^d} \left( \frac{\bm{y} \otimes \bm{y}}{|\bm{y}|^2} \right) \left( \bm{u}(\bm{x} + \bm{y}) - \bm{u}(\bm{x}) - D[\bm{u}](\bm{x}) \bm{y} \chi^{(s)}(\bm{y}) \right) K(\bm{y}) \ d\bm{y}.
\]
On the other hand, for any $\bm{v}\in \mathcal[{S}]^{d}$, 
\[
\lim_{N\to\infty , \delta \to 0}\langle\mathbb{L}\bm{u}_{N, \delta},  \bm{v}\rangle = \lim_{N\to\infty , \delta \to 0}\langle \bm{u}_{N, \delta}, \mathbb{L}^\ast\bm{v}\rangle = \langle \bm{u}, \mathbb{L}^\ast\bm{v}\rangle = \langle\mathbb{L}\bm{u}, \bm{v} \rangle
\]
and so $\mathbb{L}\bm{u}_{N, \delta}$ to $\mathbb{L}\bm{u}$ in the sense of distributions. Since the limits must be the same, we have that 
\[
-\mathbb{L}\bm{u} = \int_{\mathbb{R}^d} \left( \frac{\bm{y} \otimes \bm{y}}{|\bm{y}|^2} \right) \left( \bm{u}(\bm{x} + \bm{y}) - \bm{u}(\bm{x}) - D[\bm{u}](\bm{x}) \bm{y} \chi^{(s)}(\bm{y}) \right) K(\bm{y}) \ d\bm{y}.
\]

We are left to prove continuity of $\mathbb{L}$. 
Let $\bm{x}_k \to \bm{x}$ with $\bm{x}_k,\bm{x} \in \Omega$. We need to show $I(\bm{x}_k) \to I(\bm{x})$ as $k \to \infty$.  
To apply Dominated Convergence Theorem to $I(\bm{x}_k)$, we have to bound 
\[
\left( \frac{\bm{y} \otimes \bm{y}}{|\bm{y}|^2} \right) \left( \bm{u}(\bm{x}_k + \bm{y}) - \bm{u}(\bm{x}_k)  - D[\bm{u}](\bm{x}_k) \bm{y} \chi^{(s)}(\bm{y}) \right) K(\bm{y})
\] by an integrable function $\bm{h}(\bm{y})$ uniformly in $ \Omega$, for all $k \in \mathbb{N}$. 

For $|\bm{y}| \leq 1$, the same Hölder estimate used for the approximation yields a uniform bound 
by $C|\bm{y}|^{\epsilon - d}$, independent of $k$. 
For $|\bm{y}| > 1$, the tail term
is dominated by the sum of an $L^1$-function coming from 
$\bm{u} \in [L^1(\mathbb{R}^d, \psi)]^d$ (as in the argument above)
and by the integrable tails $C|\bm{y}|^{-d-2s}$ and $C|\bm{y}|^{-d-2s+1}\chi^{(s)}(\bm{y})$.
Hence, there exists a fixed integrable function $\bm{h}(\bm{y})$, independent of $k$, such that the integrand is bounded by $\bm{h}(\bm{y})$ for all $k$. 
Since the integrand converges pointwise to the corresponding integrand at $\bm{x}$ 
(using the continuity of $\bm{u}$ and $D[\bm{u}]$ for $s \geq 1/2$) the Dominated Convergence Theorem implies that 
\[
I(\bm{x}_k) \to I(\bm{x}) \quad \text{for all } \bm{x} \in \Omega.
\] One can iterate the above argument to conclude that if $\bm{u}\in [C^{\infty}_{b}(\mathbb{R}^{d})]$, then so is $\mathbb{L}\bm{u}$.
\end{proof}

\end{document}